\documentclass[a4paper]{article}

\usepackage[colorlinks=true, linkcolor=blue, urlcolor=blue, citecolor=blue,%
pagebackref=true]{hyperref}

\usepackage{amsthm}
\usepackage{amsmath}
\usepackage{amssymb}

\def\a{\alpha}
\def\b{\beta}
\def\d{\delta}

\def\e{\varepsilon}
\def\f{\varphi}

\def\Ga{\Gamma}

\def\l{\lambda}

\def\O{\Omega}

\def\R{{\bf R}}

\def\Z{{\bf Z}}

\def\F{{\bf F}}

\def\wh{\widehat}
\def\wt{\widetilde}

\def\HH{\mathcal{H}}

\def\fsl{\mathfrak{sl}}

\def\supp{{\rm supp}\,}

\def\be{\begin{equation}}
\def\ee{\end{equation}}
\def\bea{\begin{eqnarray}}
\def\eea{\end{eqnarray}}
\def\bean{\begin{eqnarray*}}
\def\eean{\end{eqnarray*}}

\newcommand{\lhdneqq}{\:{\substack{\raisebox{-7pt}{$\displaystyle\lhd$}%
\\ \raisebox{3pt}{$\scriptscriptstyle \neq\;$}}}\:}

\newtheorem{lem}{Lemma}
\newtheorem{thm}[lem]{Theorem}
\newtheorem{prp}[lem]{Proposition}

\newtheorem{cor}[lem]{Corollary}

\newtheorem{dfn}[lem]{Definition}

\theoremstyle{definition}

\def\mod{{\rm mod}\;}

\def\dist{{\rm dist}}
\def\Lip{{\rm Lip}}

\def\Reg{{\rm Reg}}
\def\Rego{{\rm Reg}^{\circ}}
\def\chio{\chi^{\circ}}
\def\SU{{\rm SU}}
\def\SL{{\rm SL}}
\def\PSL{{\rm PSL}}

\def\Ker{{\rm Ker}}

\def\diam{{\rm diam}}
\def\gap{{\rm gap}}
\def\dist{{\rm dist}}
\def\L{{\rm L}}
\def\Tr{{\rm Tr}}

\newcommand{\lf}{\left}
\renewcommand{\r}{\right}
\newcommand{\igap}{\:}

\title{Random walks in compact groups}
\author{P\'eter P\'al Varj\'u\thanks{
I gratefully acknowledge the support
of the Simons Foundation and the European Research Council
(Advanced Research Grant 267259)
}
}

\begin{document}

\maketitle

\begin{abstract}
Let $X_1,X_2,\ldots$ be independent identically distributed
random elements of a compact group $G$.
We discuss the speed of convergence of the law of the product
$X_l\cdots X_1$ to the Haar measure.
We give poly-log estimates for certain finite groups and for
compact semi-simple Lie groups.
We improve earlier results of Solovay, Kitaev, Gamburd, Shahshahani
and Dinai.
\end{abstract}

%%%%%%%%%%%%%%%%%%%%%%%%%%%%%%%%%%%%%%%%%%%%%%%%%%%%%%%%%%%%%%%%%%%%%%
\section{Introduction}\label{sc_intro}
%%%%%%%%%%%%%%%%%%%%%%%%%%%%%%%%%%%%%%%%%%%%%%%%%%%%%%%%%%%%%%%%%%%%%

Let $G$ be a group and $S\subset G$ a finite set.
We study the distribution of the product of $l$
random elements of $S$.
In particular, we are interested in how fast this
distribution becomes uniform as $l$ grows.

We discuss the problem in two different but very related settings:
profinite groups, and compact Lie groups.

%%%%%%%%%%%%%%%%%%%%%%%%%%%%%%%%%%%%%%%%%%%%%%%%%%%%%%%%%%%%%%%%%%%%%
\subsection{Two ways to measure uniformity}\label{sc_resfin}
%%%%%%%%%%%%%%%%%%%%%%%%%%%%%%%%%%%%%%%%%%%%%%%%%%%%%%%%%%%%%%%%%%

We begin by describing the details in the first setting.
Let $G$ be a finite group and $S\subset G$ be a finite generating
set.
For simplicity, we assume that $1\in S$.
The unit element of any group is denoted by $1$ in this paper.
Write
\[
S^l=\{g_1\cdots g_l:g_1,\ldots, g_l\in S\}
\]
for the $l$-fold product set of $S$.

The diameter of $G$ with respect to $S$ is defined by
\[
\diam(G,S)=\min\{l:S^l=G\}.
\]
The diameter is the minimal length of a product that can 
express any element of the group.
Hence it is
a (very weak) quantity to measure uniformity.

We quantify uniformity in a stronger sense, too.
To this end, we introduce the notion of random walks.
Denote by $\mu_S$ the normalized counting measure on the set $S$,
and let $X_1,X_2,\ldots$ be a sequence of independent random elements of
$S$ with law $\mu_S$.
The (simple) random walk is the sequence of random elements $Y_0,Y_1,\ldots$
of $G$ such that $Y_0=1$ almost surely, and $Y_{l+1}=X_{l+1}Y_l$
for all $l\ge 0$.

Denote by $\mu_S^{*(l)}$ the $l$-fold convolution of
the measure $\mu_S$ with itself.
Convolution of two measures (or functions) $\mu, \nu$ is defined
by the usual formula
\[
\mu*\nu(g)=\sum_{h\in G}\mu(gh^{-1})\nu(h).
\]
Observe that $\mu_S^{*(l)}$ is the law of $Y_l$.

We want to understand, how large $l$ is needed to be taken such that
$\mu_S^{*(l)}$ is ``very close" to the uniform distribution.
We make this precise with the following construction.
Consider the space $L^2(G)$ which is simply the vector space of
complex valued functions on $G$ endowed with the standard scalar product.
The group $G$ acts on this by
\[
\Reg_G(g)f(h)=f(g^{-1}h)\quad {\rm for }\quad f\in L^2(G).
\]
This is a unitary representation called the regular representation.

To a measure $\mu$, we associate an operator (linear transformation)
on $L^2(G)$:
\[
\Reg_G(\mu)=\sum_{g\in G} \mu(g)\cdot\Reg_G(g).
\]
This is an analogue of the Fourier transform of classical harmonic
analysis.
In particular, it has the following property:
\[
\Reg_G(\mu*\nu)=\Reg_G(\mu)\cdot\Reg_G(\nu),
\]
which is easy to verify from the definitions.
We can recover $\mu$ by the formula
\be\label{eq_inversion}
\mu=\Reg_G(\mu)\d_{1},
\ee
where $\d_{1}$ is the Dirac measure supported at 1.
(Recall that $G$ is finite, hence we can embed the space of probability
measures into $L^2$.)

The operator $\Reg_G(\mu_S)$ is of norm 1 and it acts trivially on
the one dimensional space of constant functions.
We denote by $\Rego_G$ the restriction of $\Reg_G$ to the 1 codimensional
space orthogonal to constants.
We define by
\[
\gap(G,S):=1-\|\Rego_G(\mu_S)\|
\]
the spectral gap of the random walk on $G$ generated by $S$.
Later, we will consider a slightly more general situation and
replace $\mu_S$ by an arbitrary probability measure $\mu$.
Then we write
\[
\gap(G,\mu):=1-\|\Rego_G(\mu)\|.
\]

It is clear that
\be\label{eq_gapapprox}
\lf\|\Reg_G\lf(\mu_S^{*(l)}\r)-\Reg_G(\mu_G)\r\|<e^{-l\cdot\gap(G,S)}.
\ee
This shows that if we take say $l=10\log|G|/\gap(G,S)$, then
$\mu_S^{*(l)}$ is very close to the uniform distribution
(see (\ref{eq_inversion})).
In particular, the support of $\mu_S^{*(l)}$ is the whole group.
Thus spectral gap is a stronger quantity to measure uniformity
than diameter.
Somewhat surprisingly we can obtain a bound in the other direction,
as well.

\begin{lem}\label{lm_folkfin}
Let $G$ be an arbitrary finite group, and $1\in S\subset G$.
Suppose that $S$ is symmetric, i.e. $g\in S$ if and only if
$g^{-1}\in S$.
We have:
\[
(\diam(G,S)-1)/\log|G|\le \gap(G,S)^{-1}\le |S|\diam(G,S)^2.
\]
\end{lem}

This lemma is well-known.
The second (and more difficult) inequality of it
can be found for example in \cite[Corollary 1 in Section 3]{DSC}.
The other estimate follows directly from (\ref{eq_gapapprox}).
The assumption on symmetricity is not an essential one.
Later, at the beginning of Section \ref{sc_thfin} we show how to reduce
the problem to the symmetric case.

We stop for a moment to connect our terminology to the computer science
and combinatorics literature.
If $S$ is symmetric, then the
matrix of the operator $\Reg_G(\mu_S)$ is proportional
to the adjacency matrix of the Cayley graph of $G$ with respect to
the generating set $S$.
In that case $\gap(G,S)$ is proportional to the spectral gap
of the Cayley graph.
If $\gap(G,S)\ge c>0$ for a family of groups and generators than
the corresponding family of Cayley graphs are called expanders.

However, in this paper we are looking for weaker bounds of
the form $\gap(G,S)\ge\log^{-A}|G|$ which, in light
of the above Lemma, is equivalent to $\diam(G,S)\le\log^{A'}|G|$
as long as say $|S|<10$ and one does not care about the value of
$A$ and $A'$.
We call such bounds poly-logarithmic.

%%%%%%%%%%%%%%%%%%%%%%%%%%%%%%%%%%%%%%%%%%%%%%%%%%%%%%%%%%%%%%%%%%%%%%
\subsection{Prior results}
%%%%%%%%%%%%%%%%%%%%%%%%%%%%%%%%%%%%%%%%%%%%%%%%%%%%%%%%%%%%%%%%%%%%%

It was conjectured by Babai and Seress \cite[Conjecture 1.7]{BS}
that the family of non-Abelian
finite simple groups have poly-logarithmic diameter, i.e.
there is a constant $A$ such that
$\diam(G,S)\le\log^{A}|G|$ holds for any non-Abelian finite simple
group $G$ and generating set $S\subset G$.
This has been verified for the family $SL_2(\F_p)$ by Helfgott
\cite[Main Theorem]{Hel1}
and for finite simple groups of Lie type of bounded rank
by Breuillard, Green and Tao \cite[Theorem 7.1]{BGT} and Pyber and Szab\'o
\cite[Theorem 2]{PSz} independently.
The conjecture is still open for other families
of finite simple groups.
The best known bound to date on the diameter of alternating groups
is due to Helfgott and Seress \cite{HS} and it is slightly weaker than
poly-logarithmic.

However, the first results on poly-logarithmic diameter were obtained
for non-simple groups.
Fix a prime $p$ and a symmetric set $S\subset \SL_2(\Z)$ such
that (the projection of) $S$ generates $\SL_2(\Z/p^2\Z)$.
Gamburd and Shahshahani \cite[Theorem 2.1]{GSh} proved
the poly-log diameter estimate
\[
\diam(\SL_2(\Z/p^n\Z),S)\le C\log^{A}|\SL_2(\Z/p^n\Z)|,
\]
where $C$ depends only on $S$
and $A$ is an absolute constant.
Dinai \cite[Theorem 1.2]{Din1}  observed that the result holds with $C$
depending only on $p$ and he also improved the parameter $A$.
Thus the family $\SL_2(\Z/p^n\Z)$ enjoys a uniform poly-log diameter
bound with respect to arbitrary generators.
Using the result of Helfgott \cite{Hel1},
the constant $C$ can be made absolute.
In \cite[Theorem 1.1]{Din2}, Dinai extended the result
to the quotients of other Chevalley groups over
local rings.

The result of this paper is also about non-simple finite groups.
In fact, it is part of our assumptions
that all simple quotients of the groups we study
has poly-logarithmic diameter.
Our result has a huge overlap with \cite{GSh}, \cite{Din1}
and \cite{Din2}.
We will remark on this later.

%%%%%%%%%%%%%%%%%%%%%%%%%%%%%%%%%%%%%%%%%%%%%%%%%%%%%%%%%%%%%%%%%%%%%%
\subsection{The role of quasirandomness}
%%%%%%%%%%%%%%%%%%%%%%%%%%%%%%%%%%%%%%%%%%%%%%%%%%%%%%%%%%%%%%%%%%%%%%

Our approach is based on representation theory, but we will use
only very basic facts of the theory.
We explain the key idea of the paper, which allows us to estimate
the spectral gap based on lower bounds for the dimension of nontrivial
representations of the group.
Let $\mu$ be a probability measure on $G$.
Write $\chi_{G}$ and $\chio_{G}$ respectively for the characters
of $\Reg_G$ and $\Rego_G$.
Furthermore, we write
\[
\chi(\mu)=\sum_{g\in G} \mu(g)\cdot\chi(g),
\]
where $\chi$ is the character of a representation of $G$.
Then $\chi_G(\mu)$ is the trace of the operator $\Reg_G(\mu)$.

For the moment, assume that $\mu$ is symmetric that is
$\mu(g^{-1})=\mu(g)$ for all $g\in G$.
Then $\Reg_{G}(\mu)$ is selfadjoint.
We can decompose $L^2(G)$ as the orthogonal sum of irreducible
subrepresentations of $\Reg_G$.
As is well known, if $\pi$ is an irreducible representation of $G$,
then exactly $\dim \pi$ isomorphic copies of $\pi$ appear
in the decomposition of $\Reg_G$.
If $\l$ is an eigenvalue of $\Reg_G(\mu)$ then there is a
corresponding eigenvector in one of the irreducible representations of
$G$.
Moreover, there is one in each isomorphic copy.
This leads us to the following inequality which is fundamental
to us:
\be\label{eq_SX}
\dim(\pi)\cdot\l^2\le\chi_G(\mu*\mu),
\ee
where $\pi$ is an irreducible representation of $G$ such that there
is an eigenvector corresponding to $\l$ in $\pi$.
(Note that all eigenvalues of $\Reg_G(\mu*\mu)$ are non-negative.)

This idea goes back to Sarnak and Xue \cite{SX}
and it is also one of the major steps in the work of
Bourgain and Gamburd \cite{BG1} on estimating spectral gaps
and also in several papers \cite{BG2}, \cite{BG3},
\cite{BGS}, \cite{Var0}, \cite{BV}, \cite{SGV} which follow it.
Gowers \cite{Gow} also exploited the idea, and introduced the
term quasirandom for groups that does not have low dimensional
non-trivial representations.
He proved several properties of such groups, in particular that
they do not have large product-free subsets.
Nikolov and Pyber \cite[Corollary 1]{NP}
pointed out that Gowers's result implies
that any element of a quasirandom group can be expressed as
the product of three elements of a sufficiently large
subset.

In what follows, $G$ is a profinite group and $\O$ denotes the family
of finite index normal subgroups of it.
An interesting example
to have in mind is $G=\SL_d(\wh\Z)$,
where $\wh\Z$ is the profinite (congruence) completion of the integers.
(The reader unfamiliar with the notion of profinite groups may
assume that $G$ is finite without loss of generality.)
Inspired by Gowers's terminology, we make the following definition.

\begin{dfn}
We say that a profinite group $G$ is
$(c,\a)$-quasirandom if for every irreducible unitary representation $\pi$
of $G$, we have
\[
\dim \pi\ge c[G:\Ker(\pi)]^\a.
\]
\end{dfn}

%%%%%%%%%%%%%%%%%%%%%%%%%%%%%%%%%%%%%%%%%%%%%%%%%%%%%%%%%%%%%%%%%%%%%%%%%
\subsection{Results about profinite groups}
%%%%%%%%%%%%%%%%%%%%%%%%%%%%%%%%%%%%%%%%%%%%%%%%%%%%%%%%%%%%%%%%%%%%%%%%%

Let $G$ be a profinite group and $\Ga$ a finite index normal subgroup.
Our plan is to prove the estimate
\be\label{eq_polylogest}
\gap(G/\Ga,S)\ge c\log^{-A}[G:\Ga]
\ee
with constants $c,A$ independent of $\Ga$ and $S$.
We prove this statement by induction as follows:
We find a larger subgroup $\Ga\lhdneqq\Ga'\in\O$
and assume that the above spectral gap estimate holds for $\Ga'$.
We use this to bound the trace of the operator
$\Reg_{G/\Ga'}(\mu_S^{*(l)})$ for a suitable integer $l$.
This in turn gives an estimate for the trace of
$\Reg_{G/\Ga}(\mu_S^{*(l)})$, and by (\ref{eq_SX})
we can estimate $\gap(G/\Ga,S)$ and continue by induction.

In order that the induction step works, we need to ensure that
$[\Ga':\Ga]$ is ``not very large" compared to $[G:\Ga]$.
Of course, we cannot always find a suitable $\Ga'$, in particular
when $G/\Ga$ is simple.

The statement of the following theorem is very technical, so we
first explain it informally:
We suppose that $G$ is a quasirandom profinite group and partition
its finite index normal subgroups into two sets: $\O_1\dot{\cup}\O_2$.
We assume that for $\Ga\in\O_1$ we can find a larger subgroup
$\Ga'\in\O$ so that $[\Ga':\Ga]$ is ``not very large".
We assume further that for $\Ga\in\O_2$, the quotient $G/\Ga$ has
poly-log spectral gap (i.e. (\ref{eq_polylogest}) is satisfied).
Then we can conclude a poly-log estimate (\ref{eq_polylogest}) for all
$\Ga\in\O$.

\begin{thm}\label{th_fin}
Let $c_1,\a,\b$ and $A$ be positive numbers satisfying
$\b<\a$ and
\[
A\ge (\a-\b)^{-1}-1.
\]
Then there is a positive number $C_1$ depending only on $c_1,\a,\b$
and $A$ such that the following holds.

Let $G$ be a $(c_1,\a)$-quasirandom profinite group.
Let $\O_1\dot{\cup}\Omega_2=\O$ be a partition of the family of
finite index normal subgroups of $G$.

Suppose that
for all $\Ga\in\O_1$, we have $[G:\Ga]>C_1$
and there is $\Ga'\in\O$ with $\Ga\lhdneqq\Ga'$ such that
\be\label{eq_conditioninduction}
[\Ga':\Ga]\le[G:\Ga']^\b.
\ee

Let $\mu$ be a Borel probability measure on $G$ and suppose that
\be{\label{eq_BS}}
\gap(G/\Ga,\mu)\ge B\cdot\log^{-A}[G:\Ga]
\ee
for all $\Ga\in\O_2$.

Then for all $\Ga\in\Omega$
\[
\gap(G/\Ga,\mu)\ge C_1^{-1}B\cdot\log^{-A}[G:\Ga].
\]
\end{thm}

To verify assumption (\ref{eq_BS}), we need to use known cases
of the Babai--Seress conjecture mentioned above.
Fortunately, this is available for many cases in
the papers \cite[Main Theorem]{Hel1}, \cite[Corollary 1.1]{Hel2},
\cite[Theorem 7.1]{BGT}, \cite[Theorem 2]{PSz}.
In particular we can deduce the following.

\begin{cor}\label{cr_sld}
Let $S\subset \SL_d(\wh\Z)$ be a set generating a dense subgroup.
Then for all integers $q>1$, we have 
\[
\gap(\SL_d(\Z/q\Z),S)\ge c\log^{-A}(q),
\]
where
$A>0$ is a number depending on $d$ and $c$ is a number depending
on $d$ and $|S|$.
\end{cor}

We stress here that $c$ depends only on the cardinality
of $S$.
Note that this dependence is necessary, owing to the following
example:
It may happen that all but one element of $S$ projects into a proper subgroup
of $\SL_d(\Z/q\Z)$ in which case the characteristic function of
the subgroup is an almost invariant function if $|S|$ is large.

In the case $S\subset \SL_d(\Z)$ one can improve the poly-log bound to
a uniform one.
More precisely, Bourgain and Varj\'u \cite[Theorem 1]{BV} proved
$\gap(\SL_d(\Z/q\Z),S)\ge c$
with a constant $c>0$ independent of $q$ but dependent on $S$.
This strongly resonates with the sate of affairs for compact semi-simple
Lie groups to be discussed below.

Finally, we compare our result to the papers of
Gamburd and Shahshahani \cite[Theorem 2.1]{GSh} and Dinai
\cite[Theorem 1.2]{Din1},
\cite[Theorem 1.1]{Din2}.
These papers give results similar to Corollary \ref{cr_sld}, and
even more, the proofs have some common features with our approach.
However, they obtain the diameter bound directly and they use
properties of commutators instead of representation theory.
A flaw of our approach that it is not constructive, i.e. we can
not give an efficient algorithm to express an element of $G/\Ga$
as a product of elements of $S$.
On the contrary, \cite[Theorem 2.1]{GSh}, \cite[Theorem 1.2]{Din1},
\cite[Theorem 1.1]{Din2} give such algorithms.
The advantage of our paper that it seems to apply in more
general situations, e.g. \cite[Theorem 2.1]{GSh}, \cite[Theorem 1.2]{Din1},
\cite[Theorem 1.1]{Din2} are restricted to the case when $q$ is
the power of a prime.

A comment on the exponents:
If one considers the group $G=\SL_2(\Z_p)$, where $p$ is a fixed
prime, then the conditions of
Theorem \ref{th_fin} can be satisfied for any $A>2$.
This via Lemma \ref{lm_folkfin}
exactly recovers the diameter bound for $\SL_2(\Z/p^k\Z)$
in the paper \cite[Theorem 1.2]{Din1}, but
our bound for the spectral gap
is better than what could be obtained from $\cite{Din1}$.
If one considers $\SL_d$ with $d$ large than our bounds
deteriorate compared to \cite{Din1}.
However, it seems possible that a more careful version of our
argument could give better bounds, but this requires a more
precise understanding of the representations.
In Section \ref{sc_remarks}  we include
some remarks about what this would require.
These ideas are worked out in the setting of compact Lie groups.

%%%%%%%%%%%%%%%%%%%%%%%%%%%%%%%%%%%%%%%%%%%%%%%%%%%%%%%%%%%%%%%%
\subsection{Results about compact Lie groups}\label{sc_rescomp}
%%%%%%%%%%%%%%%%%%%%%%%%%%%%%%%%%%%%%%%%%%%%%%%%%%%%%%%%%%%%%%%%%

We turn to the second setting of our paper. 
Let $G$ be a semi-simple compact Lie group endowed with the bi-invariant
Riemannian metric.
We denote by $\dist(g,h)$ the distance of two elements $g,h\in G$.
Let $\e>0$ be a number and $S\subset G$ be a finite
subset which generates a dense subgroup.
Again, for simplicity, we assume $1\in S$.
We define the diameter of $G$ at
scale $\e$ with respect to $S$ by
\[
\diam_\e(G,S)=\min\{l:\;{\rm for\:every\:}
g\in G\;{\rm there\: is\:} h\in S^l\;{\rm such\: that\:}
\dist(g,h)<\e\}.
\]

We also introduce the relevant spectral gap notion.
This requires some basic facts about the representation theory
of compact Lie groups.
We follow the notation in \cite{BtD}, but the results we need
can be found in many of the textbooks on the subject, as well.

Let $T$ be a maximal torus
in $G$ and denote by $\L T$ its tangent space at 1.
Then $T$ can be identified with $\L T/I$ via the exponential map,
where $I$ is a lattice in $\L T$.
Denote by $\L T^*$ the dual of $\L T$ and by $I^*\subset\L T^*$
the lattice dual to $I$.
We denote by $R\subset I^*$ the set of roots, and by $R_+$
(a choice of) the positive roots.
We fix an inner product $\langle\cdot,\cdot\rangle$ on $\L T$
which is invariant under the Weyl group.
Denote by
\[
K=\{u\in\L T:\langle u,v\rangle>0 \;{\rm for\: every}\; v\in R_+\}
\]
the positive Weyl chamber and by $\overline K$ its closure.

It is well known (see
\cite[Chapter IV (1.7)]{BtD}) that the irreducible representations
of $G$
can be parametrized by the elements of $I^*\cap\overline K$.
For $v\in I^*\cap\overline K$, we denote by $\pi_v$ the unitary
representation of $G$ with highest weight $v$.

For a finite Borel measure $\mu$ on $G$, we write
\[
\pi_v(\mu)=\int\pi_v(g)\igap d\mu(g),
\]
which is an operator (linear transformation) on the
representation space of $\pi_v$.
Let $r>1$ be a number, and $S\subset G$ a finite set
which contains 1 and
generates a dense subgroup.
Define the spectral gap at
scale $1/r$ with respect to $S$ by
\[
\gap_r(G,S):=1-\max_{0<|v|\le r}\|\pi_v(\mu_S)\|.
\]

As in the finite case, the notions of spectral gap and diameter
are closely related:

\begin{lem}\label{lm_folkcomp}
Let $G$ be a compact connected semi-simple Lie group,
and let $1\in S\subset G$ be finite.
Then there is a constant $C>0$ depending only on $G$ such that
for any $\e>0$
\[
\diam_\e(G,S)\le \frac{C\log(\e^{-1})}{\gap_{C\e^{-C}}(G,S)}
\]
and for any $r\ge1$
\[
\gap_r(G,S)\ge\frac{1}{|S|\diam_{(Cr)^{-C}}(G,S)^2}.
\]
\end{lem}

This lemma is also well-known.
We give a proof in Section \ref{sc_lip} for completeness.

In Section \ref{sc_comp}, we develop an analogue
for compact Lie groups of the ideas
explained in the previous section.
A replacement for (\ref{eq_SX}) was given by Gamburd, Jacobson and
Sarnak \cite{GJS} and it appeared in various forms in
\cite{BGSU1} and \cite{BGSU2}, as well.
However, these are based on direct calculation with characters
rather than on multiplicities of eigenvalues.
A more direct analogue of (\ref{eq_SX}) and also of the results
of Gowers \cite{Gow} and Nikolov and Pyber \cite{NP} was
developed very recently by Saxc\'e \cite{dS}.

Our main result in the setting of compact Lie groups is the
following:

\begin{thm}\label{th_comp}
For every  semi-simple compact connected Lie group $G$,
there are numbers $c,r_0$ and $A$
such that the following holds.
Let $\mu$ be an arbitrary probability measure on $G$.
Then
\[
\gap_{r}(G,\mu)\ge c\cdot\gap_{r_0}(G,\mu)\cdot\log^{-A}r.
\]
For simple groups, the value of $A$ can be found in Table \ref{tbl_A}.
For semi-simple groups $A$ is the maximum of the corresponding
values over all simple quotients of $G$.
In particular, $A\le2$ for all groups.
\end{thm}

\begin{table}[ht]
\centering
\begin{tabular}{|c| c| c| c| c| c| c| c| c|}
\hline
$A_n$ & $B_n$ & $C_n$ & $D_n$ & $E_6$ &  $E_7$ & $E_8$ & $F_4$ & $G_2$\\
\hline
$1+\frac{2}{n+1}$ & $1+\frac{1}{n}$ & $1+\frac{1}{n}$ &
$1+\frac{1}{n-1}$ & $\frac{7}{6}$ & $\frac{10}{9}$ & $\frac{16}{15}$ &
$\frac{7}{6}$ & $\frac{4}{3}$\\
\hline
\end{tabular}
\caption{The value of $A(G)$ in terms of the Dynkin diagram}
\label{tbl_A}
\end{table}

A poly-logarithmic estimate for $\diam_\e(\SU(2),S)$
was found by Solovay and Kitaev independently.
Nice expositions are in the paper of Dawson and Nielsen
\cite[Theorem 1]{DN}
and in the book \cite[Chapter 8.3]{KSV}, where they obtain the
bound $\diam_\e(\SU(2),S)<\log^{3}\e^{-1}$.
Our theorem provides the same bound for the diameter of $\SU(2)$.
On the other hand, our spectral gap bound
in Theorem \ref{th_comp} beats anything that could be obtained from
a diameter bound via
Lemma \ref{lm_folkcomp}.

Dolgopyat \cite[Theorems A.2 and A.3]{Dol}
gave an estimate for the spectral
gap which is weaker than poly-logarithmic but his argument would
give a poly-log estimate without significant changes.
His proof consists of a version of the Solovay-Kitaev argument
(that he discovered independently)
and a variant of Lemma \ref{lm_folkcomp}.
The connection between the present paper and \cite{Dol}
was pointed out to me by Breuillard, see also his survey \cite{Bre}.

We note that Bourgain and Gamburd \cite[Corollary 1.1]{BGSU1},
\cite[Theorem 1]{BGSU2}
showed that when
$\mu=\mu_S$ for some finite set $1\in S\subset \SU(d)$ and the
entries of the elements of $S$ are algebraic numbers, then
\[
\gap_r(G,S)\ge c
\]
for some constant $c$ depending on $G$ and $S$.
Their argument is likely to carry over to arbitrary semi-simple
compact Lie groups, however the assumption on algebraicity
is essential for the proof.
It is a very interesting open problem whether this assumption can
be removed.
Moreover, we raise the following question:
Is it true that there are numbers $c,r_0$ depending only on $G$
such that
\[
\gap_r(G,\mu)\ge c\cdot\gap_{r_0}(G,\mu)
\]
for all probability measures $\mu$ on $G$?

Finally, we state a technical result which almost immediately follow
from Theorem \ref{th_comp}.
Its purpose is that this is the version used in the paper \cite{Var}
to study random walks in the group of Euclidean isometries.

For a measure $\nu$ on $G$,
we define the measure $\wt\nu$ by
\[
\int f(x)\igap d\wt\nu(x)=\int f(x^{-1})\igap d\nu(x)
\]
for all continuous functions $f$.
We write $m_G$ for the Haar measure on $G$.

\begin{cor}\label{cr_comp}
Let $G$ be a compact Lie group with semi-simple connected
component.
Let $\mu$ be a probability measure on $G$ such that
$\supp(\wt\mu*\mu)$ generates a dense subgroup in $G$.
Then there is a constant $c>0$ depending only on $\mu$ such that
the following holds.
Let $\f\in \Lip(G)$ be a function such that $\|\f\|_2=1$
and $\int \f\igap  dm_G=0$.
Then
\[
\lf\|\int\f(h^{-1}g)\igap d\mu(h)\r\|_2<1-c\log^{-A}(\|\f\|_\Lip+2),
\]
where $A$ depends on $G$ and is the same as in Theorem \ref{th_comp}.
\end{cor}

The rest of the paper is organized as follows.
In Section \ref{sc_fin} we give the proof of Theorem \ref{th_fin}
and Corollary \ref{cr_sld}.
The proof of Theorem \ref{th_comp} is given in Section \ref{sc_comp}.
Sections \ref{sc_fin} and \ref{sc_comp} are independent, but there
is a strong analogy between the two arguments.
Finally we prove Corollary \ref{cr_comp} and Lemma \ref{lm_folkcomp}
in Section \ref{sc_lip}.

Throughout the paper we use the letters $c$ and $C$ to denote
numbers which may depend on several other quantities and their
value may change at each occurrence.
We follow the convention that $c$ tends to denote numbers that
we consider ``small" and $C$ denotes those that we consider ``large".

%%%%%%%%%%%%%%%%%%%%%%%%%%%%%%%%%%%%%%%%%%%%%%%%%%%%%%%%%%
\subsection{Motivation}\label{sc_motiv}
%%%%%%%%%%%%%%%%%%%%%%%%%%%%%%%%%%%%%%%%%%%%%%%%%%%%%%%%%%

In recent years there was a lot of progress on the Babai-Seress
conjecture mentioned above, although it has not been settled yet.
In addition, poly-log spectral gap and diameter is known to
hold for some families of non-simple finite groups as well.
What the scope of this phenomenon is,
is an interesting question.
Our result on finite groups is a (very modest) step towards
understanding this.
In Section \ref{sc_remarks} we include some remarks on how to
exploit our approach for non-quasirandom groups.

Our main motivation for Theorem \ref{th_comp} is the application
in the paper \cite{Var}.
Although it seems easy to extend the Solovay-Kitaev approach to 
prove similar poly-log type bounds, we believe that our method
gives better exponents, at least for spectral gaps.

There are many recent applications of spectral gaps.
Many of these require stronger bounds than what we obtain in this
paper, e.g. the results in 
\cite{BG1}, \cite{BG2}, \cite{BG3},
\cite{BGS}, \cite{Var0}, \cite{BV}, \cite{SGV}
mentioned above.
However, for some applications, the poly-log type bounds are enough,
at least to obtain the same qualitative result.
Prominent examples are the work of Ellenberg, Hall and Kowalski
\cite{EHK},
the Group Large Sieve developed by Lubotzky and Meiri \cite{LM},
the study of curvatures in Apollonian Circle Packings by Bourgain
and Kontorovich \cite{BK} and the study of random walks on Euclidean
isometries by Varj\'u \cite{Var}.
However, our results are relevant only for
the last two of the above papers, because \cite{EHK} and \cite{LM}
requires spectral gaps only for products of two simple groups.
In addition, in the case of \cite{BK} a better uniform spectral gap is
available.

\noindent {\bf Acknowledgments.}
I thank Jean Bourgain for discussions related to this project,
in particular, for explaining to me the relation between the
diameter and the spectral gap.
I thank Nicolas de Saxc\'e for careful reading of my manuscript
and for his suggestions which greatly
improved the presentation of the paper.
I thank Emmanuel Breuillard for calling my attention
on the work of Dolgopyat \cite{Dol}.

%%%%%%%%%%%%%%%%%%%%%%%%%%%%%%%%%%%%%%%%%%%%%%%%%%%%%%%%%%%%%%%%
\section{Profinite groups}
\label{sc_fin}
%%%%%%%%%%%%%%%%%%%%%%%%%%%%%%%%%%%%%%%%%%%%%%%%%%%%%%%%%%%%%%%%

%%%%%%%%%%%%%%%%%%%%%%%%%%%%%%%%%%%%%%%%%%%%%%%%%%%%%%%%%%%%%%%%
\subsection{Proof of Theorem \ref{th_fin}}\label{sc_thfin}
%%%%%%%%%%%%%%%%%%%%%%%%%%%%%%%%%%%%%%%%%%%%%%%%%%%%%%%%%%%%%%%%%

Recall that $G$ is a profinite group and $\O$ is the family of
finite index normal subgroups of it.
Assume that the hypothesis of the theorem is
satisfied.
We note that $\a\le1/2$, since $(\dim \pi)^2\le|G/\Ga|$ for any
irreducible representation $\pi$ of the group $G/\Ga$.
Consequently $A\ge 1$.

We first explain how to reduce to the case when
$\mu$ is symmetric.
Define the probability measure $\wt\mu$ by
\[
\int f(g) \igap d\wt\mu(g)=\int f(g^{-1}) \igap d\mu(g).
\]
(If $G$ is finite, this is can be expressed as $\wt\mu(g)=\mu(g^{-1})$.)
Clearly $\wt\mu*\mu$ is symmetric and
\[
\Rego_{G/\Ga}(\wt\mu*\mu)=\Rego_{G/\Ga}(\mu)^*\cdot\Rego_{G/\Ga}(\mu).
\]
Hence
\[
\|\Rego_{G/\Ga}(\wt\mu*\mu)\|=\|\Rego_{G/\Ga}(\mu)\|^2.
\]
This yields
\[
\gap({G/\Ga},\wt\mu*\mu)
\ge\gap({G/\Ga},\mu)
\ge\gap(G/\Ga,\wt\mu*\mu)/2.
\]

This shows that the spectral gap of $\wt\mu*\mu$ is roughly proportional
to that of $\mu$, hence it suffices to prove the theorem for
the first one.
{}From now on, we assume that $\mu$ is symmetric, hence the
operator $\Reg_{G/\Ga}(\mu)$ is selfadjoint and has an eigenbasis
with real eigenvalues.

We define
\[
l_\Ga=2\lfloor C_\Ga\log^{A+1}[G:\Ga]\rfloor,
\]
where
\[
C_\Gamma:=C_0(10-{\log^{-1/10}([G:\Ga])}),
\]
and
\[
C_0=\max_{\Ga\in\O_2}\frac{1}{\gap(G/\Ga,\mu)\cdot\log^{A}[G:\Ga]}.
\]
(The role of the subtracted term in the definition of $C_\Ga$
is simply to cancel lower order terms later.)

Recall that we denote by $\chi_{G/\Ga}$ the character of $\Reg_{G/\Ga}$.
Our goal is to prove
\be\label{eq_find}
\chi_{G/\Ga}\lf(\mu^{*(l_\Ga)}\r)\le M,
\ee
where $M\ge2$ is a suitably large number depending on $\a$ and $\b$.
Once we proved this, the claim of the theorem will be concluded easily.

The proof is by induction with respect to the partial order $\lhd$
on $\Ga\in\Omega$.
Suppose that (\ref{eq_find}) holds for all
$\Ga'\in\Omega$ with $\Ga\lhdneqq\Ga'$.

We distinguish two cases.
First we suppose that $\Ga\in\O_2$.
Note that $\chi_{G/\Ga}(\mu^{*(l_\Ga)})$ is the trace
of the operator $\Reg_{G/\Ga}(\mu^{*(l_\Ga)})$, hence it is
the sum of its eigenvalues.
The non-trivial eigenvalues are bounded by
$e^{-l_\Ga\cdot\gap(G/\Ga,\mu)}$,
hence
\bean
\chi_{G/\Ga}\lf(\mu^{*(l_\Ga)}\r)&\le& 1+[G:\Ga]\cdot
e^{-l_\Ga\cdot\gap(G/\Ga,\mu)}\\
&\le&1+[G:\Ga]\cdot e^{-C_0\log^{A+1}[G:\Ga]\cdot\gap(G/\Ga,\mu)}\le 1+1
\eean
using the definitions of $l_\Ga$, $C_\Ga$ and $C_0$. 
Hence (\ref{eq_find}) follows.

Now we suppose that $\Ga\notin\O_2$, hence $\Ga\in\O_1$, in particular
$[G:\Ga]\ge C_1$, where $C_1$ can be taken as large as we need
depending on $\a,\b,c_1,A$.
Let $\Ga'$ be such that
$\Ga\lhd\Ga'\in\O$
and $[\Ga':\Ga]$
is minimal but larger than 1.
Then by assumption, $[\Ga':\Ga]\le[G:\Ga']^\b$.

Since
$\chi_{G/\Ga}$ is pointwise majorized by the function
$[\Ga':\Ga]\cdot\chi_{G/\Ga'}$, and $\mu^{*(l_\Ga')}$
is a positive measure, we have
\be\label{eq_oldbound}
\chi_{G/\Ga}\lf(\mu^{*(l_{\Ga'})}\r)\le
[\Ga':\Ga]\cdot\chi_{G/\Ga}\lf(\mu^{*(l_{\Ga'})}\r)\le M[\Ga':\Ga].
\ee
We applied (\ref{eq_find}) for $\Ga'$ in the second inequality.

Denote by $\l_0=1,\l_1,\ldots \l_k$ the eigenvalues of the operator
$\Reg_{G/\Ga}(\mu)$ each listed as many times as its multiplicity.
Then
\bea
\chi_{G/\Ga}\lf(\mu^{*(l_\Ga)}\r)&=&1+\sum_{i=1}^k\l_i^{l_\Ga}%\nonumber\\
\le1+\lf(\sum_{i=1}^k\l_i^{l_{\Ga'}}\r)\cdot\max\{\l_i\}^{l_{\Ga}-l_{\Ga'}}
\nonumber\\
&\le&1+M[\Ga':\Ga]\max\{\l_i\}^{l_{\Ga}-l_{\Ga'}}.
\label{eq_newbound}
\eea
We applied (\ref{eq_oldbound}) in the last line.
Also note, that all terms are positive because $l_{\Ga}$ is even by
construction.

Our next goal is to obtain a sufficient bound on the $\l_i$.
This can be deduced from (\ref{eq_oldbound}) and the
assumption about the dimension of faithful representations.
The representation $\Reg(G/\Ga)$ can be decomposed
as the orthogonal sum
of irreducible subrepresentations.
Each irreducible representation occur with multiplicity equal to
its dimension.

Consider now an eigenvalue $\l_i$.
There is a corresponding eigenvector which
is contained in an irreducible subrepresentation of $\Reg(G/\Ga)$.
Denote this representation by $\pi$.
Write
$\Ga''=\Ker(\pi)$.

First we consider the case that $\Ga''=\Ga$.
It follows that $\l_i^{l_{\Ga'}}$ is an eigenvalue of
$\Reg_{G/\Ga}(\mu^{*(l_{\Ga'})})$ with multiplicity at least
\[
\dim\pi\ge c_1[G:\Ga]^{\a},
\]
(by the assumption on quasirandomness).
Hence by (\ref{eq_oldbound}) we can conclude that
\[
\l_i^{l_{\Ga'}}\le
\frac{M[\Ga':\Ga]}{c_1[G:\Ga]^{\a}}.
\]

Next, we consider the case when $\Ga''\neq\Ga$.
By the assumption on the minimality of $[\Ga':\Ga]$, we have
$[G:\Ga'']\le[G:\Ga']$.
We apply (\ref{eq_find}) for $\Ga''$
along with the bound for the multiplicity of eigenvalues and get
\[
\l_i^{l_{\Ga''}}\le
\frac{M}{c_1[G:\Ga'']^{\a}}.
\]
An easy calculation shows that the bound we obtain for $|\l_i|$
is worsening when $[G:\Ga'']$ grows.

Thus in both cases we obtain:
\be\label{eq_lmax}
\l_i^{l_{\Ga'}}\le
\frac{M[\Ga':\Ga]}{c_1[G:\Ga']^{\a}}.
\ee

We plug this into (\ref{eq_newbound}) and we want to conclude
$\chi_{G/\Ga}(\mu^{*(l_\Ga)})\le M$.
To this end, we need
\[
M[\Ga':\Ga]\lf(\frac{M[\Ga':\Ga]}
{c_1[G:\Ga']^{\a}}\r)^{(l_\Ga-l_{\Ga'})/l_{\Ga'}}
\le M-1
\]
that is
\be\label{eq_wanted}
\frac{M}{M-1}[\Ga':\Ga]\le\lf(\frac{c_1[G:\Ga']^{\a}}
{M[\Ga':\Ga]}\r)^{(l_\Ga-l_{\Ga'})/l_{\Ga'}}.
\ee

For simplicity, we write $X=\log[G:\Ga']$ and $Y=\log[\Ga':\Ga]$.
Then by the definition of $l_\Ga$:
\bea
\frac{l_\Ga-l_{\Ga'}}{l_{\Ga'}}
&=&\frac{2\lfloor C_\Ga(X+Y)^{A+1}\rfloor-2\lfloor C_{\Ga'}X^{A+1}\rfloor}
{2\lfloor C_{\Ga'}X^{A+1}\rfloor}\nonumber\\
&\ge& \frac{C_\Ga(X+Y)^{A+1}-C_{\Ga'}X^{A+1}-1}{C_{\Ga'}X^{A+1}}
\nonumber\\
&\ge&
\frac{C_{\Ga'}(X+Y)^{A+1}-C_{\Ga'}X^{A+1}}{C_{\Ga'}X^{A+1}}
+\frac{C_\Ga-C_{\Ga'}}{C_{\Ga'}}-\frac{1}{C_{\Ga'}X^{A+1}}
\nonumber\\
&\ge& (A+1)\frac{Y}{X}+c_2\frac{Y}{X^{11/10}},\label{eq_lcomput}
\eea
where $c_2$ is an absolute constant if $X$ is larger than an absolute
constant.
(Using the definition of $C_\Ga$, we evaluate
$(C_\Ga-C_{\Ga'})/C_{\Ga'}$ and get the second term
of (\ref{eq_lcomput}).
Notice that this is of larger order
of magnitude
than $1/X^{A+1}$.)

Then the logarithm of the right hand side of (\ref{eq_wanted})
is at least
\bean
&&\lf((A+1)\frac{Y}{X}+c_2\frac{Y}{X^{11/10}}\r)
(\a X-\log (M/c_1)-Y)\\
&&\qquad\quad\ge \a(A+1)Y-\log(M/c_1)(A+1)\frac{Y}{X}-(A+1)\frac{Y^2}{X}
+c_3\frac{Y}{X^{1/10}}\\
&&\qquad\quad\ge \a(A+1)Y-(A+1)\frac{Y^2}{X}
+c_3\frac{Y}{2X^{1/10}}
\eean
if $X$ is sufficiently large depending on $\a,\b,c_1,A$.
Here $c_3$ is a sufficiently small constant depending on $\a$ and $\b$
satisfying $c_3 X\le c_2(\a X-\log (M/c_1)-Y)$.
(Recall that $Y\le\b X$ by assumption.)

To get (\ref{eq_wanted}), we need
\[
\log(M/(M-1))+Y\le \a(A+1)Y-(A+1)\frac{Y^2}{X}
+c_3\frac{Y}{2X^{1/10}}
\]
that is
\be\label{eq_final}
A+1\ge\frac{1+\log(M/(M-1))/Y-c_4X^{-1/10}}{\a-Y/X+c_4X^{-1/10}},
\ee
where $c_4$ is yet another number depending on $\a,\b,c_1,A$.

We consider two cases.
If $Y\le\sqrt{X}$ (and $X$ is sufficiently large) then
$Y/X\le c_4X^{-1/10}$.
Clearly $Y\ge\log 2$, hence (\ref{eq_final}) holds if
\[
A+1\ge \frac{1+\log(M/(M-1))/\log 2}{\a}.
\]

On the other hand if $Y\ge\sqrt{X}$
(and $X$ is sufficiently large) then
$\log(M/(M-1))/Y\le c_4X^{-1/10}$.
Recall that $Y/X\le \b$ by assumption, hence
(\ref{eq_final}) holds if $A+1\ge(\a-\b)^{-1}$.

By choosing $M$ sufficiently large depending on $\a,\b$, we can ensure
that
\[
\frac{1+\log(M/(M-1))/\log 2}{\a}\le \frac{1}{\a-\b}.
\]
Thus (\ref{eq_final}) holds in either case if $A+1\ge(\a-\b)^{-1}$,
which was assumed in the theorem.
This completes the induction to prove (\ref{eq_find}).

Let $\l$ be an eigenvalue of $\Rego_{G/\Ga}(\mu)$.
We want to show that
\[
|\l|<1-C_1^{-1}B\cdot\log^{-A}[G:\Ga].
\]
Denote by $\pi$ an irreducible representation of $G/\Ga$ that contains
an eigenvector corresponding to $\l$.
Without loss of generality, we can assume that $\Ga=\Ker(\pi)$.
{}From quasirandomness and (\ref{eq_find}) we get
\[
\l^{l_\Ga}\le\frac{M}{c_1[G:\Ga]^\a}.
\]
Thus
\[
\log|\l|\le\frac{\log{M/c_1}-\a\log[G:\Ga]}{l_\Ga}.
\]
If we compare this with the definition of $l_\Ga$ we can conclude the
theorem.

%%%%%%%%%%%%%%%%%%%%%%%%%%%%%%%%%%%%%%%%%%%%%%%%%%%%%%%%%%%%%%%%%%%%%%%%%
\subsection{Some remarks about weakening the hypothesis on quasirandomness}
\label{sc_remarks}
%%%%%%%%%%%%%%%%%%%%%%%%%%%%%%%%%%%%%%%%%%%%%%%%%%%%%%%%%%%%%%%%%%%%%%%%%

In this section we present some ideas that lead to a refined
version of (\ref{eq_newbound}) and (\ref{eq_lmax}).
Using this, one could obtain a version of
Theorem \ref{th_fin} with a weaker hypothesis instead of
quasirandomness.
Namely one would require that the quotient groups does not have
``many" irreducible representations with ``small" dimension.
This weaker form of quasirandomness would be very technical
hence we do not state a theorem.
However, (based on analogy with compact Lie groups) it seems
possible that these ideas lead to better bounds for the
groups $\SL_d(\Z_p)$ than Theorem \ref{th_fin}.

\begin{prp}\label{pr_refined}
Let $G$ be a finite group and $N$ a normal subgroup.
Let $\rho_1,\ldots,\rho_n$ be all irreducible representations of $N$
up to isomorphism.
Denote by $a(\rho_j)$ the number of $G$-conjugates of $\rho_j$.
Denote by $d(\rho_j)$ the smallest possible dimension of a
representation
of $G$ whose restriction to $N$ contains $\rho_j$.

Let $\mu$ be a symmetric probability measure on $G$ and suppose that
\be\label{eq_indhyp}
\chi_{G/N}\lf(\mu^{*(2l)}\r)\le M
\ee
for some numbers $l$ and $M>1$.

Then for all integers $l'>l$ we have:
\be\label{eq_refined}
\chi_G\lf(\mu^{*(2l')}\r)
\le\sum_{j=1}^{n}\lf[\dim(\rho_j)^2M
\cdot\lf(\frac{a(\rho_j)\dim(\rho_j)^2M}{d(\rho_j)}\r)^{(l'-l)/l}\r].
\ee
\end{prp}

Observe that $\sum\dim(\rho_j)^2=|N|$
and $a(\rho_j)\dim(\rho_j)^2\le|N|$, hence we obtain
(\ref{eq_newbound}) combined with (\ref{eq_lmax}), if we estimate
$d(\rho_j)$ using quasirandomness.
If $d(\rho_j)$ is small only for a very few $j$, Proposition
\ref{pr_refined} significantly improves the argument given
in the previous section.

\begin{proof}
We recall some facts from representation theory.
Let $\pi$ be an irreducible representation of $G$,
and denote by $\chi_\pi$ its character.
Denote by $\pi|_N$ the restriction of $\pi$ to $N$.
By Clifford's theorem, the irreducible components of $\pi|_N$
is a $G$-conjugacy class of representations
and each appears with the same multiplicity.
We denote by $a(\pi)$ the number of different irreducible components
of $\pi|_N$, by $b(\pi)$ their common multiplicity and by $c(\pi)$
their common dimension.
Thus $\dim(\pi)=a(\pi)\cdot b(\pi)\cdot c(\pi)$.

Denote  one of the irreducible components of $\pi|_N$ by $\rho$
and  its character by $\chi_\rho$.
Let $\rho^{g_1},\cdots,\rho^{g_{a(\pi)}}$ be all $G$-conjugates
of $\rho$.
In what follows, the function
\[
\f=\sum_{i=1}^{a(\pi)}\chi_{\rho^{g_i}}
\]
will play an important role.
We also extend it to $G$ by setting it $0$ in the complement of $N$,
i.e. we write $\wt\f(g)=\f(g)$ for $g\in N$ and $\wt \f(g)=0$
otherwise.

We use two inner products, one on $L^2(G)$ and one on $L^2(N)$ defined by:
\[
\langle f_1,f_2\rangle_G=\frac{1}{|G|}\sum_{g\in G}f_1(g)f_2(g),
\quad
\langle f_1,f_2\rangle_N=\frac{1}{|N|}\sum_{g\in N}f_1(g)f_2(g).
\]
These are the inner products with respect to which the irreducible
characters of the corresponding groups are orthonormal.

We write:
\[
\langle \chi_\pi,\wt\f\rangle_G
=\frac{|N|}{|G|}\langle\chi_{\pi|_N},\f\rangle_N
=\frac{|N|a(\pi)b(\pi)}{|G|}.
\]
A similar calculation shows that the inner product of $\wt\f$
with an irreducible character of $G$ is always non-negative.

Since $\wt\f$ is a class function on $G$, it can be decomposed as
a linear combination of irreducible characters.
According to the above calculation, the coefficient of $\chi_\pi$
is $|N|a(\pi)b(\pi)/|G|$ and all other characters have non-negative
contribution.
Thus
\be\label{eq_piest1}
\chi_{\pi}\lf(\mu^{*(2l)}\r)
\le\frac{|G|}{|N|a(\pi)b(\pi)}\wt\f\lf(\mu^{*(2l)}\r).
\ee
(Note that $\chi_{\pi}(\mu^{*(2l)})\ge0$ being the trace of a positive
operator.)

On the other hand,
\[
\|\wt\f\|_\infty\le\sum_{i=1}^{a(\pi)}\|\chi_{\rho^{g_i}}\|_\infty=
a(\pi)c(\pi).
\]
Thus for every $g\in G$, we have
\[
|\wt\f(g)|\le\frac{|N|a(\pi)c(\pi)}{|G|}\chi_{G/N}(g).
\]
(Note that for $g\notin N$ both sides are 0.)
Therefore
\be\label{eq_piest2}
\wt\f\lf(\mu^{*(2l)}\r)
\le\frac{|N|a(\pi)c(\pi)}{|G|}\chi_{G/N}\lf(\mu^{*(2l)}\r).
\ee

We combine (\ref{eq_piest1}) and (\ref{eq_piest2}) and get
\[
\chi_{\pi}\lf(\mu^{*(2l)}\r)
\le\frac{c(\pi)}{b(\pi)}\chi_{G/N}\lf(\mu^{*(2l)}\r).
\]
This implies that for all eigenvalues $\l$ of $\pi(\mu^{*(2l)})$,
we have
\be\label{eq_lambdaest}
|\l|<\frac{c(\pi)}{b(\pi)}M=\frac{a(\pi)c(\pi)^2}{\dim(\pi)}M.
\ee
(Here we also used the hypothesis (\ref{eq_indhyp}).)

Now let $\pi_1,\pi_2,\ldots,\pi_k$ denote all the
irreducible representations (up to isomorphism) of $G$ whose
restriction to $N$ contain $\rho$.
By a calculation very similar to the one leading to (\ref{eq_piest1})
we get
\[
\sum_{i=1}^k\frac{|N|a(\pi_i)b(\pi_i)}{|G|}\chi_{\pi_i}\lf(\mu^{*(2l)}\r)
\le\wt\f\lf(\mu^{*(2l)}\r).
\]
Combining with (\ref{eq_piest2}), we get
\[
\sum_{i=1}^k\frac{b(\pi_i)}{c(\pi_i)}\chi_{\pi_i}\lf(\mu^{*(2l)}\r)
\le\chi_{G/N}\lf(\mu^{*(2l)}\r)\le M.
\]
(We used again the hypothesis (\ref{eq_indhyp}).)
Multiplying by $a(\pi)c(\pi)^2=a(\pi_i)c(\pi_i)^2$
(which is independent of $i$)
we get
\be\label{eq_piest3}
\sum_{i=1}^k\dim(\pi_i)\chi_{\pi_i}\lf(\mu^{*(2l)}\r)
\le a(\pi)c(\pi)^2M.
\ee

We use (\ref{eq_lambdaest}) and write for an integer $l'\ge l$:
\be\label{eq_piest4}
\sum_{i=1}^k\dim(\pi_i)\chi_{\pi_i}\lf(\mu^{*(2l')}\r)
\le a(\pi)c(\pi)^2M\lf(\frac{a(\pi)c(\pi)^2M}{\dim(\pi)}\r)^{(l'-l)/l}.
\ee
We sum (\ref{eq_piest4}) for $\rho=\rho_1,\ldots,\rho_n$ and get
(\ref{eq_refined}).
\end{proof}

%%%%%%%%%%%%%%%%%%%%%%%%%%%%%%%%%%%%%%%%%%%%%%%%%%%%%%%%%%%%%%%%%%%%%%%%
\subsection{Proof of Corollary \ref{cr_sld}}
\label{sc_sld}
%%%%%%%%%%%%%%%%%%%%%%%%%%%%%%%%%%%%%%%%%%%%%%%%%%%%%%%%%%%%%%%%%%%%%%%%

We first discuss quasirandomness.
This was already proved for $\SL_d(\wh\Z)$ by Bourgain and Varj\'u
\cite{BV}.
In fact, it is easy to deduce it from the corresponding
result about $\SL_d(\Z_p)$ which was obtained by
Bourgain and Gamburd
\cite[Lemma 7.1]{BG2} for $d=2$ and by Saxc\'e \cite[Lemme 5.1]{dS}
for $d\ge 2$.
For completeness, we explain this deduction.

For an integer $q>0$ denote by
\[
\Ga_q=\{g\in \SL_d(\wh\Z): g\equiv 1\; (\mod\, q)\}
\]
the mod $q$  congruence subgroup of $\SL_d(\wh\Z)$.

Let $p$ be a prime, $k\ge 1$  be an integer,
and let $\pi$ be a representation of $\SL_d(\wh\Z)/\Ga_{p^{k}}$
which is not a representation of $\SL_d(\wh \Z)/\Ga_{p^{k-1}}$.
By \cite[Lemme 5.1]{dS}, $\pi$ is of dimension
at least
\[
c\cdot[\SL_d(\wh\Z):\Ga_{p^k}]^{1/(d+1)},
\]
where
$c>0$ is a number depending on $d$.
For any $\e>0$, we can replace this bound by
$|\SL_d(\wh\Z)/\Ga_q|^{1/(d+1)-\e}$
if $p$ is sufficiently large
depending on $\e$.

Let $\pi$ be an irreducible unitary representation of $\SL_d(\wh \Z)$.
Since $\Ga_q$, $q>1$ form a system of neighborhoods of $1$ in 
$\SL_d(\wh \Z)$, there is $q$ such that $\Ker(\pi)\supset\Ga_q$.
Let $q$ be minimal with this property.
Let $q=p_1^{k_1}\cdots p_n^{k_n}$ such that $p_i$ are primes.
Then
\[
\SL_d(\wh\Z)/\Ga_q=
\lf(\SL_d(\wh\Z)/\Ga_{p_1^{k_1}}\r)\times\cdots
\times\lf(\SL_d(\wh\Z)/\Ga_{p_n^{k_n}}\r).
\]
Any representation of this group is a tensor product of representations
of $\SL_d(\wh\Z)/\Ga_{p_i^{k_i}}$.
Hence
\be\label{eq_QR}
\dim\pi\ge c^m[\SL_d(\wh\Z):\Ga_{p^k}]^{1/(d+1)-\e},
\ee
where $m$ is the number of not large
enough primes in the sense of the previous paragraph.
Thus $\SL_d(\wh\Z)$ is $(c^m,1/(d+1)-\e)$-quasirandom.

We refer the interested reader to the paper of Kelmer and
Silberman \cite[Section 4]{KS}, where
quasirandomness is proved
with optimal parameter $\a$ for some other arithmetic groups.

We define $\O_1$ and $\O_2$.
We fix an integer $M$ that we will set later depending on $d$.
Let $\Ga\lhd\SL_d(\wh\Z)$ be a finite index normal subgroup.
Denote by $q$ the smallest integer such that $\Ga_q\lhd\Ga$.
We put $\Ga$ in $\O_1$, if $q$ has at least $M+1$ prime factors
(taking multiplicities into account) and $q$ is sufficiently large.
We put $\Ga$ in $\O_2$ otherwise.

Let $\Ga\in\O_1$ and let $q$ be the same as above.
Let $p|q$ be the smallest prime divisor of $q$.
By simple calculation:
\[
[\Ga_{q/p}:\Ga_q]\le[\SL_d(\wh \Z):\Ga_{q/p}]^{1/M}.
\]
Now we define $\Ga':=\Ga_{q/p}\Ga$.
Clearly
\[
[\Ga':\Ga]\le[\Ga_{q/p}:\Ga_q],
\]
so we only need to estimate $[\SL_d(\wh \Z):\Ga']$
in terms of $[\SL_d(\wh \Z):\Ga_{q/p}]$.
For our purposes the very crude bound
\[
[\SL_d(\wh \Z):\Ga']\ge c^m[\SL_d(\wh \Z):\Ga_{q/p}]^{1/(d+1)-\e}
\]
is sufficient which follows from (\ref{eq_QR}).
This implies
\[
[\Ga':\Ga]\le [\SL_d(\wh \Z):\Ga']^{(d+2)/M}
\]
if $q$ is sufficiently large (and $\e$ is sufficiently small).

We choose $\a$ and $\b$ in such a way that $\b<\a<1/(d+1)$.
Then the quasirandomness is satisfied, and also
(\ref{eq_conditioninduction}) if we set $M\ge(d+2)/\b$.

It is left to verify (\ref{eq_BS}) for $\mu=\mu_S$.
First we note that by the same argument as at the beginning of
the proof of Theorem \ref{th_fin}, we can assume that $S$ is symmetric.
We show that the groups $\SL_d(\wh \Z)/\Ga_q$ for $\Ga_q\in\O_2$
have poly-logarithmic diameter with respect to any generating set $S$.
In light of Lemma \ref{lm_folkfin} this implies (\ref{eq_BS}).

Let now $\Ga_q\in\O_2$. 
There are two possibilities.
If $q$ is small (e.g. $q\le C_1$ or as in the definition of $\O_1$),
then we have the trivial bound
\[
\diam(\SL_d(\wh \Z)/\Ga_q,S)\le|\SL_d(\wh \Z)/\Ga_q|
\le C\log^A(\SL_d(\wh \Z)/\Ga_q)
\]
for some suitably large constant $C$.

The other situation that may happen is that $q$ contains at most $M$
prime factors counting multiplicities.
In this case we can easily deduce the poly-log diameter bound from
\cite[Theorem 7.1]{BGT} and \cite[Theorem 2]{PSz} which
contain this result in the case when
$q$ is prime.
This deduction is very similar to \cite[Proof of Proposition 3]{BV}.

Let $q_0=1,q_1,q_2,\ldots,q_n=q$ be a sequence of integers such that
$q_{i+1}/q_i$ is a prime number for all $i$.
We will apply the following lemma repeatedly to prove the diameter
bound we are looking for.

\begin{lem}\label{lm_diam}
Fix $i$ and write $p=q_i/q_{i-1}$.
Then
\[
\diam(\SL_d(\Z/{q_i}\Z),S)
\le C(\diam(\SL_d(\Z/q_{i-1}\Z),S)+\diam(\SL_d(\Z/p\Z),S)),
\]
where $C$ is a number depending on $d$.
\end{lem}

\begin{proof}
Let
\[
D=\max\{\diam(\SL_d(\Z/q_{i-1}\Z),S),\diam(\SL_d(\Z/p\Z),S)\}.
\]
Then
\[
S^D\cdot\Ga_{q_{i-1}}=\SL_d(\Z/q_i\Z).
\]

Since $S$ is generating, $S^{D+1}$ must intersect some
$\Ga_{q_{i-1}}$-coset in at least two points.
Thus there is
\[
1\neq g_0\in S^{2D+1}\cap\Ga_{q_{i-1}}.
\]

If $p\nmid q_{i-1}$, and hence
$\Ga_{q_{i-1}}/\Ga_{q_i}=\SL_d(\Z/p\Z)$,
we also want to show that $g_0$ can be
taken non-central in $\SL_d(\Z/p\Z)$.
With the same argument as above, we can show that
$S^{(j+1)D+j}$ intersects all $\Ga_{q_{i-1}}$-cosets in at least
$j+1$ points.
Taking $j=|Z(SL_d(\Z/p\Z))|$, we can find a suitable $g_0$
in $S^{(j+1)D+j}$.

We put
\[
X=\{g^{-1}g_0g,g^{-1}g_0^{-1}g:g\in S^D\}.
\]
We show that
\be\label{eq_conjclaim}
X^{C}=\Ga_{q_{i-1}}/\Ga_{q_i}
\ee
for some constant $C$ depending on $d$.

We have two cases.
First, we suppose that $p\nmid q_{i-1}$.
Then $\Ga_{q_{i-1}}/\Ga_{q_i}=\SL_d(\Z/p\Z)$, and
$X$ is a non-trivial conjugacy class.
In this case (\ref{eq_conjclaim}) is a result of Lev \cite[Theorem 2]{Lev}.

Now suppose that $p|q_{i-1}$.
In this case $\Ga_{q_{i-1}}/\Ga_{q_i}$ is isomorphic to
$\fsl_d(\Z/p\Z)$, and the conjugation action
\[
h\mapsto g^{-1}hg,\quad g\in\SL_d(\Z/q_i\Z),\;h\in\Ga_{q_{i-1}}/\Ga_{q_i}
\]
factors through $G/\Ga_p=\SL_d(\Z/p\Z)$.
Now the claim (\ref{eq_conjclaim}) follows from \cite[Lemma 5]{BV}.

Therefore, we have $S^{C\cdot D}\supset\Ga_{q_{i-1}}/\Ga_{q_i}$
for some other constant $C$ depending on $d$.
This implies $S^{(C+1)\cdot D}=\SL_d(\Z/q_{i}\Z)$ which was to be proved.
\end{proof}

Now Corollary \ref{cr_sld} is immediate.
By \cite[Theorem 7.1]{BGT} and \cite[Theorem 2]{PSz} we have
\[
\diam(\SL_d(\Z/p\Z),S)\le C\log^{A}(p)
\]
for all primes $p$.
We can use Lemma \ref{lm_diam} repeatedly to show
\[
\diam(\SL_d(\wh\Z)/\Ga_q,S)\le C\log^{A}(p)
\]
for $\Ga_q\in\O_2$,
where $p$ is the largest prime factor of $q$ and $C$ is a different
constant depending on $M$.
This is precisely the poly-log diameter estimate we were looking for.

%%%%%%%%%%%%%%%%%%%%%%%%%%%%%%%%%%%%%%%%%%%%%%%%%%%%%%%%%%%%%%%%%%%%%%
\section{Compact Lie groups}
\label{sc_comp}
%%%%%%%%%%%%%%%%%%%%%%%%%%%%%%%%%%%%%%%%%%%%%%%%%%%%%%%%%%%%%%%%%%%%%%

The purpose of this section is to prove Theorem \ref{th_comp}.
Recall the definitions of $T,\L T, \L T^*, I, I^*,R, R_+,K$
and $\pi_v$
from Section \ref{sc_intro}.
Let $\mu$  be a probability measure on $G$.
Denote by $\chi_v$ the character of $\pi_v$.
For a continuous function $f$ and a
measure $\nu$ on $G$, we write
\[
f(\nu)=\int f(x)\igap d\nu(x).
\]
If $f$ is a continuous class function on $G$, then by the Peter-Weyl
theorem, we can decompose it as a linear
combination of irreducible characters.
Denote by $m_v(f)$ the coefficient of $\chi_v$ in this decomposition.

We introduce two partial orders on the space of continuous
class functions on $G$.
We write $f_1\le f_2$, if $f_1(g)\le f_2(g)$ for every $g\in G$.
We write $f_1\sqsubseteq f_2$ if $m_v(f_1)\le m_v(f_2)$ for all
$v\in \overline K\cap I^*$.
Denote by $\preceq$ the transitive closure of the union of $\le$
and $\sqsubseteq$, i.e. we write $f_1\preceq f_2$ if there is
a sequence of class functions $\f_i$ such that
\[
f_1\le\f_1\sqsubseteq \f_2\le \f_3\sqsubseteq\ldots\le \f_n\sqsubseteq f_2.
\]

These relations have a crucial property contained in the
following Lemma.
Recall that for a measure $\nu$ on $G$,
we define the measure $\wt\nu$ by
\[
\int f(x)\igap d\wt\nu(x)=\int f(x^{-1})\igap d\nu(x)
\]
for all continuous functions $f$.
We say that $\nu$ is symmetric if $\nu=\wt\nu$.

\begin{lem}\label{lm_prec}
Let $\nu$ be a symmetric probability measure and let
$f_1\preceq f_2$ be two continuous class functions on $G$.
Then
\[
f_1(\nu*\nu)\le f_2(\nu*\nu).
\] 
\end{lem}

\begin{proof}
Clearly, it is enough to prove the statements for $\le$ and $\sqsubseteq$
in place of $\preceq$.
For $\le$ it easily follows from the definitions and from the fact that
$\nu*\nu$ is a positive measure.

Suppose $f_1\sqsubseteq f_2$.
Observe that
\[
f_i(\nu*\nu)=\sum_{v\in \overline K\cap I^*}
m_v(f_i)\chi_v(\nu*\nu).
\]
Hence the claim follows from $m_v(f_1)\le m_v(f_2)$ once we prove
that $\chi_v(\nu*\nu)\ge0$ for all $v$.
This follows from $\chi_v(\nu*\nu)=\Tr(\pi_v(\nu*\nu))$
and from the fact that $\pi_v(\nu*\nu)=\pi_v(\nu)\cdot\pi_v(\nu)^*$
is a positive selfadjoint operator.
\end{proof}

Now we explain the strategy of the proof.
First of all, we note that by the argument at the beginning
of Section \ref{sc_thfin} we can assume that $\mu$
is symmetric.
Hence Lemma \ref{lm_prec} applies for
$\nu=\mu^{*(2l)}$
for all positive integers $l$.

We write for $r\ge 1$
\[
\chi_r=\lf(\sum_{|v|\le r}\chi_v\r)^2\cdot r^{-\dim \L T}
\]
which plays the role of $\chi_{G/\Ga}$ used in the previous section.
We also write
\[
l_r=2\lfloor C_r\log^{A+1} r\rfloor,
\]
where
\[
C_r=C_0\lf(10-{\log^{-1/10}(r)}\r)
\]
and $C_0$ is a suitably large constant to be set later.

Our goal is to prove the inequality
\be\label{eq_compgoal}
\chi_r\lf(\mu^{*(l_r)}\r)\le E
\ee
for some constant $E$ independent of $r$.
This will easily imply the theorem.

We assume that (\ref{eq_compgoal}) holds for some range $1\le r\le r_1$.
(This can be verified easily for $r_1=r_0$ if $C_0$ is suitably large
in terms of $\gap_{r_0}(\mu)$.)
And then we show that (\ref{eq_compgoal}) also holds for a suitable
$r=r_2$.
Iterating this argument, we can prove the claim for all $r$.

We prove the ``induction step" in the following way.
We find suitable functions $\f_1,\ldots,\f_n$
such that
\[
\chi_{r_2}\sqsubseteq\f_1+\ldots+ \f_n.
\]
Then it will be enough to estimate $\f_i(\mu^{*(l_{r_2})})$.
We will show that $\f_i\preceq B_{i}\chi_{r_1}$,
where $B_i$ is a number depending on $i,r_1$ and $r_2$.
This allows us to estimate $\f_i(\mu^{*(l_{r_1})})$.
We will also show that $m_v(\f_i)$ is either ``large" or 0,
and this will yield an estimate on the eigenvalues of 
$\pi_v(\mu^{*(l_{r_1})})$ for all $v$ which contributes to $\f_i$.
Finally this allows us to get a refined estimate on 
$\f_i(\mu^{*(l_{r_2})})$.

To implement the above plan we need methods to estimate
$m_v(f)$.
In our examples $f$ will always be the character of a representation
which is obtained from other representations using tensor products.
Explicit formulas are available for $m_v(f)$ in such cases, however,
they do not seem very practical for our purposes.
Instead, we will use elementary methods to estimate these coefficients
based on double-counting dimensions.
However, we still need some very basic facts about the representations
$\pi_v$.

The first fact is Weyl's dimension formula
\cite[{{Chapter VI. (1.7) (iv)}}]{BtD}:
\[
\dim\pi_v=\prod_{u\in R_+}
\frac{\langle u,v+\rho\rangle}{\langle u,\rho\rangle},
\]
where
\[
\rho=\frac{1}{2}\sum_{u\in R_+} u
\]
is the half sum of the positive roots.

The second fact is the content of the following lemma which
bounds the highest weight of possible irreducible constituents
of $\pi_v\otimes\pi_u$.
Recall that we denote the Haar measure on $G $ by $m_G$.
\begin{lem}
\label{lm_possible}
There is a constant $D$ depending only on $G$ such that
if
\be\label{eq_possible}
\int \chi_v\chi_u\overline{\chi_w}\igap dm_G\neq0
\ee
for some $v,u,w\in \overline K\cap I^*$, then
$|v-w|<D|u|$.
\end{lem}
\begin{proof}
If (\ref{eq_possible}) holds then $\pi_w$ is contained in
$\pi_v\otimes\pi_u$.
By \cite[Chapter VI, Lemma (2.8)]{BtD}, $v+u$ dominates
$w$ that is
\[
\langle v+u,t\rangle\ge\langle w,t\rangle
\]
for every $t\in \overline K$.
Similarly, if  (\ref{eq_possible}) holds then $\pi_v$ is contained in
$\pi_w\otimes\overline{\pi_u}$, hence
\[
\langle w+\overline{u},t\rangle\ge\langle v,t\rangle.
\]
Here $\overline{u}$ is the highest weight of $\overline{\chi_u}$.

Now let $t_1,\ldots, t_{\dim T}$ be a basis of $\L T^*$ consisting
of unit vectors in $\overline{K}$.
By the above inequalities, we have
\[
|\langle w-v,t_j\rangle|\le |u|
\]
for all $1\le j\le \dim T$.
The claim follows from this with the constant $D$ equal to the length
of the longest vector in the set
\[
\{x\in \L T^*:|\langle x,t_j\rangle|\le 1\}.
\]
\end{proof}

We proceed by some Lemmata which bound the multiplicities of some
irreducible constituents in certain tensor products.

\begin{lem}
\label{lm_mult1}
We have
\[
\int \lf(\sum_{|v|\le r}\chi_v\r)^2\igap  dm_G
=|\{v\in \overline{K}\cap I^*: |v|\le r\}|.
\]
In particular
\[
c\le m_{\chi_0}(\chi_r)\le C,
\]
where $c,C>0$ are constants depending on $G$ 
\end{lem}
\begin{proof}
Since characters form an orthonormal basis, we have
\[
\int \lf(\sum_{|v|\le r}\chi_v\r)^2\igap  dm_G=
\int \sum_{|v|\le r}\chi_v\overline{\chi_v}\igap  dm_G
=|\{v\in \overline{K}\cap I^*: |v|\le r\}|.
\]
For the second statement notice that
\[
m_{\chi_0}(\chi_r)=\frac{1}{r^{\dim\L T}}
\int \lf(\sum_{|v|\le r}\chi_v\r)^2\chi_0\igap  dm_G
=\frac{|\{v\in \overline{K}\cap I^*: |v|\le r\}|}{r^{\dim\L T}}.
\]
\end{proof}

\begin{lem}
\label{lm_mult2}
There is a constant $C>0$ depending on $G$
such that the following holds.
Let $u\in \overline{K}\cap I^*$, and $r\ge 1$.
\[
m_{\chi_u}(\chi_r)\le C\dim \chi_u.
\]
Moreover, $m_{\chi_u}(\chi_r)=0$ if $|u|>Cr$.
\end{lem}
\begin{proof}
For the first part of the lemma we write
\bean
r^{\dim \L T}m_{\chi_u}(\chi_r)
&=&\int \lf(\sum_{|v| \le r}\chi_v\r)^2\overline{\chi_u}\igap  dm_G\\
&\le&2\sum_{\substack{|v|,|w| \le r\\\dim\chi_v\le\dim\chi_w}}\int 
\chi_v\chi_w\overline{\chi_u} \igap dm_G\\
&=&2\sum_{|v|\le r}
\lf[\sum_{\substack{|w|\le r\\\dim\chi_w\ge\dim\chi_v}}
m_{\overline{\chi_w}}(\chi_v\overline{\chi_u})\r].
\eean

The sum of the multiplicities of some irreducible components
of $\pi_v\otimes\overline{\pi_u}$ can not be bigger than
the dimension of $\pi_v\otimes\overline{\pi_u}$ divided by
the minimal dimension of the irreducible components we consider.
Thus
\[
r^{\dim \L T}m_{\chi_u}(\chi_r)
\le2\sum_{|v|\le r}\frac{\dim(\chi_v\overline{\chi_u})}
{\dim(\chi_v)}
\le Cr^{\dim \L T}\dim\chi_u.
\]

The second part follows immediately from Lemma \ref{lm_possible}.
\end{proof}

For $u\in\overline{K}\cap I^*$ and $r\ge1$,
we write
\[
\psi_{u,r}=\chi_r\cdot\sum_{w:|u-w|<3Dr}\chi_w.
\]
We show that $\chi_z$ is contained in $\psi_{u,r}$ with high multiplicity
if $|u-z|\le Dr$.

\begin{lem}\label{lm_mult3}
Let $r\ge1$ and $z,u\in\overline K\cap I^*$ with $|z-u|\le Dr$.
Then
\[
m_{\chi_z}(\psi_{u,r})\ge c\cdot\frac{\dim(\chi_z)\cdot r^{\dim G}}
{\max_{{w:|u-w|<3Dr}}\{\dim (\chi_w)\}},
\]
where $c>0$ is a constant depending only on $G$.
\end{lem}
\begin{proof}
We can write
\bea
r^{\dim \L T}m_{\chi_z}(\psi_{u,r})
&=&\sum_{w:|u-w|<3Dr}\lf[
\sum_{|v_1|,|v_2|\le r}
\int\chi_{v_1}\chi_{v_2}\chi_{w}\overline{\chi_z}\igap dm_G\r]\nonumber\\
&=&\sum_{w:|u-w|<3Dr}
m_{\chi_w}\lf(\sum_{|v_1|,|v_2|\le r}
\overline{\chi_{v_1}}\,\overline{\chi_{v_2}}\chi_z\r)\nonumber\\
&\ge&\frac{\sum_{|v_1|,|v_2|\le r}
\dim{\chi_{v_1}}\cdot\dim{\chi_{v_2}}\cdot\dim\chi_z}
{\max_{w:|u-w|<3Dr}\dim\chi_w}.\label{eq_mult3}
\eea
In the last line we used the fact that all irreducible components
of $\overline{\chi_{v_1}}\,\overline{\chi_{v_2}}\chi_z$ has
highest weight $w$ satisfying $|w-z|\le 2Dr$
and hence $|u-w|<3Dr$ which follows from two
applications of Lemma \ref{lm_possible}.
Hence all possible irreducible component appears in the range of
summation, and the inequality follows by comparing dimensions.

Let $K'$ be a closed convex cone strictly contained
in the positive Weyl chamber $K$.
For $v\in K'\cap I^*$ and $|v|\ge r/2$ it follows from
Weyl's dimension formula that
$\dim(\chi_v)\ge c r^{|R_+|}$ for some constant $c>0$ depending only on
$G,K'$.
Thus the numerator in (\ref{eq_mult3}) is bounded below by
\[
c\dim(\chi_z)\cdot r^{2\dim \L T+2|R_+|}.
\]
The proof is finished by noting that $\dim G=\dim\L T+2|R_+|$.
\end{proof}

We continue implementing our plan described above.
Recall that  $r_2>r_1\ge 1$, are numbers that we will choose later
and that
we assume 
\be\label{eq_cmpgoal}
\chi_r\lf(\mu^{*(l_r)}\r)\le E
\ee
for $r\le r_1$ and for some number $E$ to be chosen later, as well.
Our goal is to prove (\ref{eq_cmpgoal}) with $r=r_2$.

Let $u_0=0,u_1,\ldots u_m$ be a maximal $Dr_1$ separated subset
of $\{v\in\overline K\cap I^*:|v|\le Cr_2\}$, where $C$
is the constant from Lemma \ref{lm_mult2}.
For $i=0,\ldots, m$, let
\[
M_{i}=\frac{C\max_{w:|u_i-w|<3Dr_1}\{\dim\chi_w\}}{cr_1^{\dim G}},
\]
where $c$ is the constant from Lemma \ref{lm_mult3} and $C$
is as above.
Write
\[
\f_i=M_i\sum_{w:|u_i-w|<Dr_1}m_{\chi_w}(\psi_{u_i,r_1})\chi_w,
\]
i.e. we removed from $\psi_{u_i,r_1}$ those irreducible
components whose multiplicities we cannot bound below.
Hence $m_{\chi_w}(\f_i)\ge C\cdot\dim(\chi_w)$ if $|w-u_i|<Dr_1$
by Lemma \ref{lm_mult3} (applied with $r=r_1$).

It follows from Lemma \ref{lm_mult2} (applied with $r=r_2$)
that
\[
\chi_{r_2}\sqsubseteq\sum_{i=0}^m\f_i.
\]
Moreover, we have
\[
\f_i\sqsubseteq M_i\psi_{u_i,r_1}\le
\lf[M_i\sum_{w:|u_i-w|<3Dr_1}\dim(\chi_w)\r]\cdot\chi_{r_1},
\]
because $\chi_{r_1}$ is non-negative, and
\[
\sum_{w:|u_i-w|<3Dr_1}\chi_w(g)\le
\sum_{w:|u_i-w|<3Dr_1}\dim\chi_w
\]
for every $g\in G$.
Clearly
\[
\sum_{w:|u_i-w|<3Dr_1}\dim\chi_w\le Cr_1^{\dim \L T}
\max_{w:|u_i-w|<3Dr_1}\{\dim\chi_w\},
\]
hence
\be\label{eq_fiest}
\f_i\preceq C\frac{(\max_{w:|u_i-w|<3Dr_1}\{\dim\chi_w\})^2}
{r_1^{2|R_+|}}\cdot\chi_{r_1}.
\ee

Denote by $N_i$ the number of positive roots $v\in R_+$ such that
$\langle u_i,v\rangle\le 4Dr_1|v|$.
Then it follows from Weyl's dimension formula that
\be\label{eq_minmax}
\max_{w:|u_i-w|<3Dr_1}\dim\chi_w\le C r_1^{N_i}
\min_{w:|u_i-w|<3Dr_1}\dim\chi_w.
\ee

After these preparations, we can give an estimate on
$\f_i(\mu^{*(l_{r_2})})$.
This is done in the next two Lemmata.

\begin{lem}\label{lm_fiest1}
Fix $0\le i\le m$ such that $N_i<|R_+|$.
If $r_1$ is sufficiently large and
\[
\log^{1/3} r_1\le \log r_2-\log r_1\le \log^{1/2} r_1,
\]
then
\[
\f_i\lf(\mu^{*(l_{r_2})}\r)\le \lf(\frac{r_1}{r_2}\r)^{(A-1)(|R_+|-N_i)}.
\]
\end{lem}

Recall that the value of $A$ is given in the statement of Theorem
\ref{th_comp} and it also appears in the definition of $l_r$ above.

\begin{proof}
For notational simplicity, write
\[
X=\max_{w:|u_i-w|<3Dr_1}\dim\chi_w,
\]
and note that by Weyl's dimension formula, we have
\be\label{eq_Xest}
X\le C r_1^{N_i}r_2^{|R_+|-N_i}.
\ee

By (\ref{eq_fiest}), the induction hypothesis and Lemma \ref{lm_prec},
we have
\be\label{eq_fiest1}
\f_i\lf(\mu^{*(l_{r_1})}\r)\le CE\frac{X^2}{r_1^{2|R_+|}}.
\ee

Denote by $\l_{\max}$ the maximum of the absolute values of the
eigenvalues of the operators $\pi_v(\mu)$ for $|v-u_i|\le D{r_1}$,
i.e. for the irreducible characters contained in $\f_i$.
Clearly
\be\label{eq_lmaxest1}
\l_{\max}^{l_{r_1}}\le CE\frac{X^2}{r_1^{2|R_+|}m_{\chi_v}(\f_i)},
\ee
where $\chi_v$ is the character of the representation, which contains
$\l_{\max}$.
Recall that by Lemma \ref{lm_mult3} and the definition of $\f_i$, we have
\[
m_{\chi_v}(\f_i)\ge c\dim\chi_v.
\]
If we combine this with (\ref{eq_minmax}) and (\ref{eq_lmaxest1}),
we get
\be\label{eq_lmaxest2}
\l_{\max}^{l_{r_1}}\le CE\frac{X}{r_1^{2|R_+|-N_i}}
\le CE\frac{r_2^{|R_+|-N_i}}{r_1^{2|R_+|-2N_i}}.
\ee
(For the second inequality we used (\ref{eq_Xest}).)

By (\ref{eq_fiest1}), we clearly have
\bea
\f_i\lf(\mu^{*(l_{r_2})}\r)&\le& CE\frac{X^2}{r_1^{2|R_+|}}\cdot
\l_{\max}^{l_{r_2}-l_{r_1}}\nonumber\\
&\le& CE\frac{r_2^{2|R_+|-2N_i}}{r_1^{2|R_+|-2N_i}}\cdot
\lf(CE\frac{r_2^{|R_+|-N_i}}
{r_1^{2|R_+|-2N_i}}\r)^{\frac{l_{r_2}-l_{r_1}}{l_{r_1}}}\label{eq_fiest2}.
\eea

By a computation very similar to (\ref{eq_lcomput}), we get
\be\label{eq_lest}
\frac{l_{r_2}-l_{r_1}}{l_{r_1}}
\ge \lf(A+1+\frac{c}{\log^{1/10} r_1}\r)
\frac{\log r_2-\log r_1}{\log r_1},
\ee
where $c$ is an absolute constant.

Now we write
\be\label{eq_lest2}
(\log r_2-2\log r_1)\frac{\log r_2-\log r_1}{\log r_1}
=(\log r_1 -\log r_2)\lf(1-\frac{\log r_2-\log r_1}{\log r_1}\r).
\ee
and
\be\label{eq_lest3}
\lf(A+1+\frac{c}{\log^{1/10} r_1}\r)
\lf(1-\frac{\log r_2-\log r_1}{\log r_1}\r)
\ge A+1+\frac{c}{2\log^{1/10}r_1}
\ee
which follows from our assumption
$\log r_2-\log r_1\le \log^{1/2} r_1$.

Combining (\ref{eq_lest}), (\ref{eq_lest2}) and (\ref{eq_lest3})
we get
\[
\lf(\frac{r_2^{|R_+|-N_i}}
{r_1^{2|R_+|-2N_i}}\r)^{\frac{l_{r_2}-l_{r_1}}{l_{r_1}}}
\le\lf(\frac{r_1}{r_2}\r)^{(A+1)(|R_+|-N_i)+c\cdot\log^{-1/10}r_1}.
\]
If we plug this into (\ref{eq_fiest2}), we get
\[
\f_i\lf(\mu^{*(l_{r_2})}\r)\le
\lf(\frac{r_1}{r_2}\r)^{(A-1)(|R_+|-N_i)+c\cdot\log^{-1/10}r_1}
(CE)^{l_{r_2}/l_{r_1}}.
\]

Finally we note that by $\log r_2-\log r_1\ge\log^{1/3} r_1$, we have
\[
\lf(\frac{r_2}{r_1}\r)^{c/\log^{1/10}r_1}
\ge e^{c\log^{7/30} r_1}\ge(CE)^2
\]
if $r_1$ is sufficiently large depending on $G$ and $E$.
On the other hand, by the computation leading to (\ref{eq_lest}),
we have $l_{r_2}/l_{r_1}<2$.
This finishes the proof.
\end{proof}

\begin{lem}\label{lm_fiest2}
Fix $0\le i\le m$ such that $N_i=|R_+|$.
If $\log r_2-\log r_1>\log^{1/3}r_1$ and $r_1$ is sufficiently large
then
\[
\f_i(\mu^{*(l_{r_2})})<C,
\]
where $C$ is a constant depending only on $G$ (and not on $E$).
\end{lem}
\begin{proof}
Let $v\in \overline{K}\cap I^*$, such that $m_{\chi_v}(\f_i)>0$.
By Weyl's formula, we have $\dim(\chi_v)\le Cr_1^{|R_+|}$, since
$N_i=|R_+|$.
Let $\l$ be an eigenvalue of $\pi_v(\mu)$.
Then by (\ref{eq_lmaxest1})
we get
\[
\l^{l_{r_1}}\le \frac{CE}{\dim \chi_v}.
\]

In fact, we can get a better bound if $|v|\le r_1$.
By a similar argument and using the induction hypothesis
for $r=|v|\le r_1$, we get
\[
\l^{l_{|v|}}\le \frac{CE}{\dim \chi_v}.
\]
Weyl's dimension formula gives $\dim\chi_v\ge c|v|$,
and an easy calculation shows that
\be\label{eq_lmax10}
\l^{l_{r_1}}\le \lf(\frac{CE}{|v|}\r)^{l_{r_1}/l_{|v|}}\le \frac{CE}{r_1}
\ee
if $|v|$ is sufficiently large (depending on $G$ and $E$).

We suppose that $r_0$ is so large that $|v|<r_0$ for those $v$
which are too small for the above argument.
We set $C_0>\gap_{r_0}^{-1}(G,\mu)$, hence for $r_0 \ge |v|\neq 0$ we have
\[
\l^{l_{r_1}}\le e^{-\gap_{r_0}(G,\mu)l_{r_1}}\le e^{-\log^{A+1}(r_1)},
\]
which is stronger than (\ref{eq_lmax10}).

By (\ref{eq_fiest}), the induction hypothesis and Lemma \ref{lm_prec},
we have
\[
\f_i\lf(\mu^{*(l_{r_1})}\r)\le CE,
\]
hence
\[
\f_i\lf(\mu^{*(l_{r_1})}\r)\le m_{\chi_0}(\f_i)
+CE\lf(\frac{CE}{r_1}\r)^{\frac{l_{r_2}-l_{r_1}}{l_{r_1}}}.
\]
Since $\log r_2-\log r_1>\log^{1/3}r_1$ and
\[
\frac{l_{r_2}-l_{r_1}}{l_{r_1}}>c\frac{\log r_2-\log r_1}{\log r_1}
\]
we easily get
\[
\f_i\lf(\mu^{*(l_{r_1})}\r)\le m_{\chi_0}(\f_i)+1
\]
if $r_1$ is sufficiently large (depending on $G$ and $E$)
which was to be proved.
\end{proof}

It is left to estimate the number of $u_i$ for which $N_i$
takes a particular value. This is done with the help of the next lemma.

Denote by $S$ the set of simple roots.
(This is not to be confused with the generating set of the
random walk, which is denoted by $S$ in other sections.)

\begin{lem}\label{lm_roots}
Let $S'\subsetneq S$ be a set of simple roots.
Denote by $R'\subset R_{+}$ the set of all positive roots which can
be expressed as a combination of elements of $S'$.
We have
\[
\frac{|S|-|S'|}{|R_+|-|R'|}\le A(G)-1,
\]
where the value of $A(G)$ is given in Table \ref{tbl_A},
for simple Lie groups and for non-simple ones
it is defined to be $A(G)=\max \{A(H)\}$ where $H$ runs through all simple
quotients.
\end{lem}
\begin{proof}
If the Dynkin diagram of $G$ is not connected, we can write
$R_+=R_1\cup\ldots\cup R_n$, where $R_i$ is a system of positive
roots in a root system with connected diagram.
Clearly
\[
\frac{|S|-|S'|}{|R_+|-|R'|}\le\max_{1\le i\le n}
\frac{|S\cap R_i|-|S'\cap R_i|}{|R_i|-|R'\cap R_i|}.
\]
Hence we can assume without loss of generality that the diagram of
$G$ is connected.

By a simple calculation, one can verify that
$|S|/|R_{+}|=A(G)-1$ as given in Table \ref{tbl_A}.
Now notice that $R'$ is itself (the set of positive roots in)
a root system and its diagram is the
subgraph spanned by $S'$ in the diagram of $G$.
Examining Table \ref{tbl_A} it is easy to check that
\[
\frac{|S'|}{|R'|}\ge A(G)-1=\frac{|S|}{|R_{+}|}.
\]
(In fact, it is enough to check that the
value in the table is never smaller for connected subdiagrams.)
Then
\[
\frac{|S|-|S'|}{|R_+|-|R'|}\le
\frac{|S|-|S||R'|/|R_+|}{|R_+|-|R'|}
=\frac{|S|}{|R_{+}|}=A(G)-1.
\]
\end{proof}

Let $S'\subset S$  be a subset of the simple roots.
We write $\O(S')$ for the set of indices $i$ such that
$\langle u_i,v\rangle\le 4Dr_1$ for $v\in S$ if and only if
$v\in S'$.

We estimate $|\O(S')|$.
Since $S'$ consist of linearly independent vectors, the elements
of $\O(S')$ are in a $Cr_1$ neighborhood of a subspace of $\L T$
of dimension $|S|-|S'|$.
Since they are $Dr_1$-separated, we have
\be\label{eq_sizeO}
|\O(S')|\le C\lf(\frac{r_2}{r_1}\r)^{|S|-|S'|}.
\ee

All positive roots are positive linear combinations of simple roots,
hence $N_i\le|R'|$ for $i\in\O(S')$, where $R'$ is the set of positive
roots which are combinations of the elements of $S'$.
If $S'\neq S$, then Lemmata \ref{lm_fiest1} and \ref{lm_roots}
together with (\ref{eq_sizeO}) gives
\[
\sum_{i\in \O(S')}\f_i\lf(\mu^{*(l_{r_2})}\r)\le C
\]
with a constant $C$ depending on $G$.
If $S'=S$, the same follows from Lemma \ref{lm_fiest2}.
Summing this up for all $S'\subset S$, we get
\be\label{eq_findend}
\sum_{i}\f_i\lf(\mu^{*(l_{r_2})}\r)\le C.
\ee
This completes the proof of (\ref{eq_cmpgoal}) for $r=r_2$
with $E=C$, where
$C$ is the constant in (\ref{eq_findend}).

We explain how to set the various parameters and how to complete
the induction.
We set $E=C$ with the constant $C$ from (\ref{eq_findend})
in the previous paragraph.
Then we pick $r_0$
to be sufficiently large (depending on $E$ and $G$)
so that all of the above arguments are
valid with $r_1\ge r_0$.

For $r_0\ge r\ge 1$, we have
\bea
\chi_r\lf(\mu^{*(l_r)}\r)
&\le& m_{\chi_0}(\chi_r)+e^{-l_r\cdot\gap_{Cr}(G,\mu)}
\cdot\sum_{0<|v|\le Cr} m_{\chi_v}(\chi_r)\nonumber\\
&\le& C+Cr^{-C_0\cdot\gap_{Cr}(G,\mu)}
\cdot\sum_{0<|v|\le Cr} \dim(\chi_v)\nonumber\\
&\le& C+Cr^{C(\dim\L T+|R_+|)-C_0\cdot\gap_{Cr}(G,\mu)}\label{eq_smallr}.
\eea
Here we first used the definition of $\gap_r(G,\mu)$,
then the definition of $l_r$ and Lemmata \ref{lm_mult1} and
\ref{lm_mult2}, finally Weyl's dimension formula.

We put
$C_0=C\cdot\gap_{Cr_0}^{-1}(G,\mu)$,
where $C$ is a suitable constant depending on the
constant in (\ref{eq_smallr}) such that
\[
\chi_r\lf(\mu^{*(l_r)}\r)\le E
\]
for $1\le r\le r_0$.
(Recall that the only constraint we had for $C_0$ above is
 in the proof of Lemma \ref{lm_fiest2} and it is satisfied with this
choice.)

Thus we see that (\ref{eq_cmpgoal}) hold for $1\le r\le r_0$.
Once we know that (\ref{eq_cmpgoal}) holds on an interval $r\in [1,a]$,
we can extend it to $r\in [1,e^{\log(a)+\log^{1/2}(a)}]$.
This follows from the above argument with the choice
$r_2=r$ and any $r_1\le a$ which satisfies
\[
\log^{1/3}r_1\le\log r_2-\log r_1\le\log^{1/2}r_1.
\]
If we apply this repeatedly, we can conclude that (\ref{eq_cmpgoal})
holds for all $r\ge1$.

Finally, we conclude the proof of the theorem.
Fix $r\ge 1$ and
suppose that $\pi_{v_0}$ is the representation for which the maximum
in the definition of $\gap_r(G,\mu)$ is attained.
We use Lemma \ref{lm_mult3} with $z=u=v_0$ and (\ref{eq_cmpgoal}) to get
\bean
\chi_{v_0}\lf(\mu^{*(l_r)}\r)
&\le& C\frac{\max_{w:|v_0-w|<3Dr}\{\dim (\chi_w)\}}
{\dim(\chi_{v_0})\cdot r^{\dim G}}\cdot\psi_{v_0,r}\lf(\mu^{*(l_r)}\r)\\
&\le& CE\frac{(\max_{w:|v_0-w|<3Dr}\{\dim (\chi_w)\})
\cdot\sum_{w:|v_0-w|<3Dr}\dim(\chi_w)}
{\dim(\chi_{v_0})\cdot r^{\dim G}}.
\eean
We evaluate dimensions using Weyl's formula and get
\[
\lf\|\pi_{v_0}\lf(\mu^{*(l_r)}\r)\r\|
\le\chi_{v_0}\lf(\mu^{*(l_r)}\r)
\le CE/\dim(\chi_{v_0})
\le CE/r\le r^{-1/2},
\]
if $r\ge r_0$ and $r_0$ is sufficiently large, as we may assume.
This implies
\[
\gap_{r}(G,\mu)\ge(1/10) \frac{\log r}{l_r}.
\]
Inspecting the definition of $l_r$ and the above choice of $C_0$,
we see that this is exactly what was to be proved.

%%%%%%%%%%%%%%%%%%%%%%%%%%%%%%%%%%%%%%%%%%%%%%%%%%%%%%%%%%%%%%
\section{Some technicalities}\label{sc_lip}
%%%%%%%%%%%%%%%%%%%%%%%%%%%%%%%%%%%%%%%%%%%%%%%%%%%%%%%%%%%%

We begin this section by proving Corollary \ref{cr_comp}.
First we give a lemma which will be used for reducing the
problem to the connected case:

\begin{lem}
\label{lm_connected}
Let $G$ be a Lie group
and let $\mu$ be a symmetric probability measure
on it such that $\supp{\mu}$
generates a dense subgroup and $1\in\supp\mu$.
Write $G^{\circ}$ for the connected component of $G$ and let
$n=[G:G^\circ]$.
Then $\supp(\mu^{*(2n-1)})\cap G^\circ$ generates a dense subgroup
of $G^{\circ}$.
\end{lem}

\begin{proof}
Suppose to the contrary
that for some $h\in G^\circ$ and $\e>0$
there is no $h'$ in the group generated
by $\supp(\mu^{*(2n-1)})\cap G^\circ$ such that $\dist(h,h')<\e$.
On the other hand, by assumption, there is $g=g_1\cdots g_l$
with $g_i\in\supp(\mu)$ and $\dist(h,g)<\e$.

We show that $g$ is in the group generated by
$\supp(\mu^{*(2n-1)})\cap G^\circ$, a contradiction.
If $l\le 2n-1$, then by definition, $g\in\supp(\mu^{*(2n-1)})\cap G^\circ$.
Suppose $l>2n-1$.
By the pigeon hole principle, there are $i\le j\le n$ such that
$g_i\cdots g_j\in G^\circ$.
We can write $g=g'\cdot g''$, where
\bean
g'&=&g_1\cdots g_{i-1}g_i\cdots g_jg_{i-1}^{-1}\cdots g_{1}^{-1}
\quad {\rm and}\\
g''&=&g_1\cdots g_{i-1} g_{j+1}\cdots g_l.
\eean
Now clearly $g',g''\in G^\circ$, and
$g'\in\supp(\mu^{*(2n-1)})$, since
it is a product of length at most $2n-1$.
Since the length of $g''$ is strictly less than that of $g$ the proof
can be completed by induction.
\end{proof}

If $G$ is a compact connected semi-simple Lie group,
the space $L^2(G)$ can be decomposed as an orthogonal sum of
finite dimensional irreducible representations.
We write $\HH_r\subset L^2(G)$ for the sum of those constituents
which have highest weight $v$ with $|v|\le r$.
In the next lemma we construct an approximate identity in $\HH_r$.

\begin{lem}\label{lm_approxid}
Let $G$ be a compact connected semi-simple Lie group.
Then for each $r$, there is a non-negative function $f_r\in\HH_r$
such that
\[
\int f_r\igap  dm_G=1, \quad
\|f_r\|_2\le Cr^{\dim G/4} \quad{\rm and}\quad
\int f_r(g)\dist(g,1)\igap  dm_G(g)\le C/\sqrt r, 
\]
where $C$ is a constant depending on $G$.
\end{lem}
\begin{proof}
We fix a maximal torus $T\subset G$.
Let $\pi$ be a faithful (not necessarily irreducible) finite dimensional
unitary representation of $G$ with real character $\chi$.
We can decompose the representation space as the sum of weight spaces,
i.e. there is an orthonormal basis $\f_1,\ldots,\f_m$
(where $m=\dim \pi$) such that the following holds.
For each $\f_i$, there is a weight $u_j\in I^*$ such that for
the elements $g\in T$ we have
$\pi(g)\f_j=e^{2\pi i\langle\log g, u_j\rangle}\f_j$.
(Here $\log: T\to \L T$ is a branch of the inverse of the exponential map.)
Since $\chi$ is real, we have
\[
\chi(g)=\sum_{j=1}^m e^{2\pi i\langle\log g, u_j\rangle}
=\sum_{j=1}^m \cos(2\pi\langle\log g, u_j\rangle).
\]

Since $\pi$ is faithful, $\chi(g)<m$ for $g\neq1$ and we can deduce
from the above formula that
\be\label{eq_chiest}
m-c_1\dist(g,1)^2\le \chi(g)\le m-c_2\dist(g,1)^2,
\ee
where $c_1,c_2>0$ are constants depending on $G$.

Denote by $r_0$ the length of the highest weight of the irreducible
components of $\pi$.
We define
\[
f_r(g)=c_r(\chi(g)+m)^{\lfloor r/r_0\rfloor},
\]
where $c_r$ is a normalizing constant so that $\int f_r\igap  dm_G=1$.
Observe that $f_r\in\HH_r$.
By simple calculation based on (\ref{eq_chiest}),
we get
\[
c \frac{\lfloor r/r_0\rfloor^{\dim G/2}}{(2m)^{\lfloor r/r_0\rfloor}}
\le c_r\le C \frac{\lfloor r/r_0\rfloor^{\dim G/2}}
{(2m)^{\lfloor r/r_0\rfloor}},
\]
where $c,C>0$ are numbers depending on $G$.

We have
\[
\|f_r\|_\infty=c_r(2m)^{\lfloor r/r_0\rfloor}\le Cr^{\dim G/2}.
\]
The $L^2$ bound now follows from
$\|f_r\|_2^2\le\|f_r\|_1\|f_r\|_\infty$.

Now using again (\ref{eq_chiest}), we get
\bean
&&\!\!\!\!\!\!\!\!\!\!\!\!\!\!\!\!\!\!\!\!\!\!\!\!\!\!\!\!\!\!\!
\int f_r(g)\dist(1,g)\igap dm_G(g)\\
&\le&
c_r\int (2m-c_2\dist(g,1)^2)^{\lfloor r/r_0\rfloor}\dist(1,g)\igap dm_G(g)\\
&\le& C \lfloor r/r_0\rfloor^{\dim G/2}
\int e^{-c_3\dist(g,1)^2\lfloor r/r_0\rfloor}\dist(1,g)\igap dm_G(g)\\
&\le& C \lfloor r/r_0\rfloor^{\dim G/2}
\int_{\R^{\dim G}}
e^{-c_3|x|^2\lfloor r/r_0\rfloor}|x|\igap dx\\
&=& C \lfloor r/r_0\rfloor^{-1/2}
\int_{\R^{\dim G}}
e^{-c_3|y|^2}|y|\igap dy
\le C/\sqrt{r},
\eean
which was to be proved.
\end{proof}

\begin{proof}[Proof of Corollary \ref{cr_comp}]
Write $\Reg(g)f(h)=f(g^{-1}h)$ for $f\in L^2(G)$,
which is the left regular representation
of $G$.

Assume to the contrary that $f\in \Lip (G)$, $\int f=0$, $\|f\|_2=1$
and yet
\be\label{eq_corcontra}
\lf\|\int \Reg(g)f\igap d\mu(g)\r\|_2=\|\Reg(\mu)f\|_2
\ge 1-c_0\log^{-A}(\|f\|_{\Lip}+2)
\ee
with a constant $c_0$ which will be chosen to be sufficiently small
depending on $\mu$.

By the same argument as in the beginning of Section
\ref{sc_fin}, we have
$\|\Reg(\mu)f\|_2^2=\|\Reg(\wt\mu*\mu)f\|_2$.
Thus (\ref{eq_corcontra}) holds (with a different $c_0$) for
$\mu$ replaced by $\wt\mu*\mu$, hence we can assume that
$\mu$ is symmetric and $1\in\supp\mu$.

Write $n=[G:G^\circ]$ as in Lemma \ref{lm_connected}.
Furthermore, we define $\mu_1$ to be the probability measure
on $G^\circ$ that we obtain from $\mu^{*(2n-1)}$ by restricting it to
$G^\circ$ and then normalizing it.
By Lemma \ref{lm_connected}, we can apply Theorem \ref{th_comp}
for the measure $\mu_1$.
Now we need to find a suitable test function related to $f$.

First, we want to rule out the possibility that $f$ is ``almost constant"
on cosets of $G^\circ$.
It follows from (\ref{eq_corcontra}) that there is a set $X\subset G$
with $\mu(X)>1-\e$ such that
\be\label{eq_Xprop}
\|f-\Reg(g)f\|_2<1/(10n)
\ee
where
$\e>0$ is as small as we wish, if we choose $c_0$
in (\ref{eq_corcontra}) sufficiently small.
(In fact, we could obtain a much stronger estimate from
(\ref{eq_corcontra})).
In particular, we can ensure that $X\cdot G^{\circ}$
generates $G/G^\circ$.

Since $\int f\igap dm_G=0$,
\[
\int \langle\Reg(g)f,f\rangle \igap d m_G(g)=0.
\]
Hence there is $g\in G$ such that $\langle\Reg(g)f,f\rangle\le 0$,
in particular $\|\Reg(g)f-f\|_2\ge\sqrt2$.
We can write $g=h_1\cdots h_ng_0$, where $h_i\in X$ and
$g_0\in G^\circ$.
By the triangle inequality, either
\be\label{eq_case1}
\|\Reg(h_1\cdots h_ng_0)f-\Reg(h_1\cdots h_n)f\|_2>1
\ee
or
\be\label{eq_case2}
\|\Reg(h_1\cdots h_j)f-\Reg(h_1\cdots h_{j-1})f\|_2>(\sqrt2-1)/n
\ee
for some $1\le j\le n$.
Since $\Reg$ is unitary, the second case yields
$\|\Reg(h_j)f-f\|_2>(\sqrt2-1)/n$ which is a contradiction to
(\ref{eq_Xprop}).
Thus only the first case is possible, from which we conclude
$\|\Reg(g_0)f-f\|_2>1$ again by unitarity.

This shows that $f$ can not be ``almost constant" on all cosets of
$G^\circ$.
Let $x_1,\ldots,x_n\in G$ be a system of representatives
for $G^\circ$-cosets.
Write $f_i$ for the restriction of $f$ to the coset $G^\circ \cdot x_i$
considered a function on $G^\circ$.
More formally:
\[
f_i(g)=f(gx_i)\in L^2(G^\circ).
\]

We have
\[
\|\Reg(g_0)f-f\|_2^2=\sum_{i=1}^n\|\Reg(g_0)f_i-f_i\|_2^2,
\]
hence there is $1\le i_0\le n$
such that 
$\|\Reg(g_0)f_{i_0}-f_{i_0}\|_2\ge 1/\sqrt n$.

We define
\[
\f=\frac{f_{i_0}-\int f_{i_0}\igap  dm_G}
{\|f_{i_0}-\int f_{i_0} \igap dm_G\|_2}
\]
aiming to estimate $\|\Reg(\mu_1)\f\|_2$ using Theorem \ref{th_comp}.

It follows from the above considerations that
$\|f_{i_0}-\int f_{i_0}\igap  dm_G\|_2\ge1/(2\sqrt n)$.
Thus $\|\f\|_\Lip<2\sqrt{n}\|f\|_\Lip$.

{}From Lemma \ref{lm_connected}, we know that $\supp{\mu_1}$ is
not contained in a proper closed subgroup of $G^\circ$.
Thus we have
\[
\gap_r(G^\circ,\mu_1)>0
\]
for all $r$.
Then Theorem \ref{th_comp} implies that
\[
\gap_r(G^\circ,\mu_1)>c\log^{-A(G)} r
\]
with a constant $c>0$ depending on $\mu$.

To apply this spectral gap estimate, we need to approximate $\f$
by a function in $\HH_r$ with small $r$.
We use Lemma \ref{lm_approxid} with $r=D(\|\f\|_\Lip+2)^4$; we
will chose the sufficiently large number $D$ later.
Then the Lemma gives:
\[
\|f_{r}*\f-\f\|_\infty
\le\int f_{r}(g)\dist(g,1)\|\f\|_{\Lip}\igap dm_G(g)
\le \frac{C}{D^{1/2}(\|\f\|_\Lip+2)}.
\]

Clearly $f_{r}*\f\in \HH_r$, and moreover
\[
\int f_{r}*\f\igap dm_G=0.
\]
Thus
\bean
\|\Reg(\mu_1)\f\|_2&\le& \|\Reg(\mu_1)(f_{r}*\f)\|_2
+\frac{C}{D^{1/2}(\|\f\|_\Lip+2)}\\
&\le&1-\gap_{r}(G^\circ,\mu_1)
+\frac{C}{D^{1/2}(\|\f\|_\Lip+2)}\\
&\le&1-c\log^{-A(G)}(D^{1/2}(\|\f\|_\Lip+2))+
\frac{C}{D^{1/2}(\|\f\|_\Lip+2)}.
\eean
Now we choose $D$ sufficiently large depending on $\mu$ so that
for the quantities in the last line we have
\[
c\log^{-A(G)}(D^{1/2}(\|\f\|_\Lip+2))
\ge \frac{2C}{D^{1/2}(\|\f\|_\Lip+2)}.
\]
This in turn implies with a different constant $c$:
\[
\|\Reg(\mu_1)\f\|_2\le1-c\log^{-A(G)}(\|\f\|_\Lip+2).
\]

Furthermore, by the definition of $\f$, we have
\[
\|\Reg(\mu_1)f_{i_0}\|_2\le\|f_{i_0}\|_2-c\log^{-A(G)}(\|f\|_\Lip+2).
\]
This yields
\[
\|\Reg(\mu_1)f\|_2\le 1-c\log^{-A(G)}(\|f\|_\Lip+2).
\]
By selfadjointness
\[
\|\Reg(\mu)f\|_2^{2n-1}\le\|\Reg(\mu^{*(2n-1)})f\|_2
\le 1-  c'(1- \|\Reg(\mu_1)f\|_2),
\]
where $c'=\mu^{*(2n-1)}(G^\circ)$ is the normalizing constant
in the definition of $\mu_1$.
This is a contradiction to (\ref{eq_corcontra}),
if we choose there $c_0$ to be sufficiently small.
\end{proof}

The rest of the section is devoted to the proof of Lemma \ref{lm_folkcomp}.
We start with a Lemma which provides an estimate on the Lipschitz
norm of a function in $\HH_r$.

\begin{lem}\label{lm_lip}
Let $G$ be a connected compact semi-simple Lie group and
let $f\in\HH_r$ with $\|f\|_2=1$.
Then $\|f\|_\Lip\le Cr^{\dim G/2+1}$,
where $C$ is a constant depending on $G$.
\end{lem}
\begin{proof}
We fix a maximal torus $T$ and write
\[
\|f\|_{\Lip(T)}=\max_{t\in T,g\in G}
\lf\{\frac{f(tg)-f(g)}{\dist(1,t)}\r\}
\]
for $f\in\Lip(G)$.
Since every element of $G$ is contained in a maximal torus, it is
enough to bound this new semi-norm which a priori could be smaller.

As in the proof of Lemma \ref{lm_approxid},
we decompose $\HH_r$ as the sum of weight spaces for
$T$.
There is an orthonormal basis $\f_1,\ldots,\f_m\in\HH_r$
(where $m=\dim \HH_r$) such that for $g\in T$ we have
\[
\pi(g)\f_j=e^{2\pi i\langle\log g, u_j\rangle}\f_j,
\]
where $u_j\in I^*$ is a weight and $|u_j|\le r$.
Thus
$\|\f_j\|_{\Lip(T)}\le Cr$.
We can write $f=\sum\a_j\f_j$ such that $\sum \a_j^2=1$.
Then
\[
\|f\|_{\Lip(T)}\le Cr\sum \a_j\le Cr\sqrt m.
\]
By Weyl's dimension formula it follows that each irreducible
representation in $\HH_r$ is of dimension at most $Cr^{|R_+|}$
and each appears with multiplicity equal to its dimension.
The number of irreducible components is at most $Cr^{\dim T}$.
Putting these together we get $m\le Cr^{\dim G}$ which proves
the Lemma.
\end{proof}

\begin{proof}[Proof of Lemma \ref{lm_folkcomp}]
First we bound the diameter in terms of the spectral gap.
Let $\e>0$ be a number, and $g_0\in G$.
Suppose that for some integer $l$, there is no $g\in S^l$
such that $\dist(g,g_0)\le\e$.

We fix a number $D$ that we will specify later, and 
write $r=D\e^{-2\dim G-2}$ and let $f_r$ be the approximate identity
constructed 
in Lemma \ref{lm_approxid}.
Then
\bea
\mathop{\int}_{\dist(g,g_0)<\frac{\e}{2}}\!\!\!\!\!\!
\Reg(\mu_S)^l f_r(g)\igap dm_G(g)\!
&=&\!\!\mathop{\int}_{\dist(g,g_0)<\frac{\e}{2}}
\sum_{h\in S^l}\mu_S^{*(l)}(h)f_r(h^{-1}g)\igap dm_G(g)\nonumber\\
&\le&\mathop{\int}_{\dist(g,1)>\frac{\e}{2}}f_r(g)\igap dm_G(g)%\nonumber\\
\le \frac{C}{\e \sqrt r}\label{eq_C}.
\eea
For the inequality between the first and second lines, we used that
$\dist(h,g)\ge\e/2$, hence
$\dist(1,h^{-1}g)\ge\e/2$.

On the other hand, we have
\[
\|1-\Reg(\mu_S)^l f_r\|_2\le C r^{\dim G/4}e^{-l\cdot\gap_r(G,S)}.
\]
(Recall $\|f_r\|_2\le C r^{\dim G/4}$ from Lemma \ref{lm_approxid}.)
Hence
\bea
\mathop{\int}_{\dist(g,g_0)<\frac{\e}{2}}\Reg(\mu_S)^l f_r(g)\igap dm_G(g)
&\ge&\mathop{\int}_{\dist(g,g_0)<\frac{\e}{2}}1\igap dm_G(g) 
-\frac{C r^{\dim G/4}}{e^{l\cdot\gap_r(G,S)}}\nonumber\\
&\ge& c\e^{\dim G}\label{eq_c}.
\eea
provided
\[
l>\frac{\log\lf(2Cr^{\dim G/4}/(c\e^{\dim G})\r)}{\gap_r(G,S)}
\ge C'\cdot\frac{\log(\e^{-1})}{\gap_r(G,S)},
\]
where $C'$ depends on $G$ and $D$.
Now we choose $D$ such that
\[
\frac{C}{\e \sqrt r}\le c\e^{\dim G},
\]
where $C$ and $c$ are the constants form (\ref{eq_C}) and (\ref{eq_c}),
respectively.
This is impossible.
We can conclude
\[
\diam(G,S)\le \frac{C\log(\e^{-1})}{\gap_{D\e^{-2\dim G-2}}(G,S)}.
\]

Now we estimate the spectral gap in terms of the diameter.
This argument was communicated to me by Jean Bourgain.
Let $r>0$ be a number, and set $\e=Dr^{-\dim G/2-1}$,
where $D>0$ depends on $G$ and will be set later.
Let $f\in\HH_r$ and assume that $\|f\|_2=1$ and $\int f=0$.
Then $\int \langle\Reg(g)f,f\rangle \igap dm_G(g)=0$, hence there is
$g\in G$ such that $\langle\Reg(g)f,f\rangle\le0$.
Thus $\|\Reg(g)f-f\|_2\ge\sqrt2$.

Let $l=\diam_{\e}(G,S)$ and $g_0=g_1\cdots g_l\in S^l$ such that
$\dist(g,g_0)\le\e$.
By Lemma \ref{lm_lip} we have
\bean
\|\Reg(g_0)f-f\|_2&\ge&\|\Reg(g)f-f\|_2-\|\Reg(g_0)f-\Reg(g)f\|_2\\
&\ge& \sqrt2 - \e Cr^{\dim G/2+1}\ge 1
\eean
if we choose $D$ to be sufficiently small in the definition of $\e$.

By the triangle inequality, there is $1\le j\le l$
such that
\[
\|\Reg(g_j)f-f\|_2=\|\Reg(g_1\cdots g_j)f-\Reg(g_1\cdots g_{j-1})f\|_2
\ge 1/l.
\]
This implies
\[
\|\Reg(g_j)f+f\|_2\le2-1/l^2.
\]
Finally, we can conclude
\bean
\|\Reg(\mu_S)f\|_2
&\le&\frac{1}{|S|}\left(\|\Reg(g_j)f+f\|_2+\left\|
\sum_{g\in S\backslash\{1,g_j\}}\Reg(g)f\right\|_2\right)\\
&\le&1-\frac{1}{|S|\diam_{\e}(G,S)^2}
\eean
which was to be proved.
(Recall the assumption $1\in S$.)
\end{proof}

\bigskip

\noindent{\sc Centre for Mathematical Sciences,
Wilberforce Road, Cambridge CB3 0WA,
England}

\noindent{\em e-mail address:} pv270@dpmms.cam.ac.uk

\bigskip

\noindent and
\bigskip

\noindent
{\sc The Einstein Institute of Mathematics, Edmond J. Safra Campus,
Givat Ram, The Hebrew University of Jerusalem, Jerusalem, 91904, Israel}

\end{document}